\documentclass[final,12pt]{elsarticle}
\usepackage{srcltx}
\usepackage{eurosym}
\usepackage{mathtools}
\usepackage{amsmath}
\usepackage{cases}
\usepackage{amsfonts}
\usepackage{amssymb}
\usepackage{amsthm}
\usepackage[utf8]{inputenc}
\usepackage{chngcntr}
\usepackage{apptools}
\usepackage{graphicx}

\usepackage[dvipsnames]{xcolor}

\usepackage{graphicx}
\usepackage{mathrsfs}
\usepackage{xcolor}
\usepackage{xpatch}
\usepackage{exscale}
\usepackage{latexsym}
\usepackage{tikz}
\usepackage{caption}
\usepackage{environ}
\usetikzlibrary{arrows.meta}

\usepackage{cancel,ulem}
\xpatchcmd{\MaketitleBox}{\hrule}{}{}{}% remove first horizontal rule (above abstract)
\xpatchcmd{\MaketitleBox}{\hrule}{}{}{}
\usepackage[colorlinks,plainpages=true,pdfpagelabels,hypertexnames=true,colorlinks=true,pdfstartview=FitV,linkcolor=blue,citecolor=red,urlcolor=black]{hyperref}
\PassOptionsToPackage{unicode}{hyperref}
\PassOptionsToPackage{naturalnames}{hyperref}
\usepackage{enumerate}
\usepackage[shortlabels]{enumitem}
\usepackage{bookmark}
\usepackage{wasysym}
\usepackage{esint}
\usepackage[ddmmyyyy]{datetime}
\usepackage[margin=2cm]{geometry}
\makeatletter
\g@addto@macro\normalsize{%
	\setlength\abovedisplayskip{4pt}
	\setlength\belowdisplayskip{4pt}
	\setlength\abovedisplayshortskip{4pt}
	\setlength\belowdisplayshortskip{4pt}
}
\numberwithin{equation}{section}
\everymath{\displaystyle}
\usepackage[capitalize,nameinlink]{cleveref}
\crefname{section}{Section}{Sections}
\crefname{subsection}{Subsection}{Subsections}
\crefname{condition}{Condition}{Conditions}
\crefname{hypothesis}{Hypothesis}{Conditions}
\crefname{assumption}{Assumption}{Assumptions}
\crefname{lemmaa}{Lemma}{Lemmas}
\crefname{lemma}{Lemma}{Lemmas}
\crefname{definition}{Definition}{Definitions}
\crefname{figure}{figure}{figures}
\crefname{question}{Question}{Questions}

\crefformat{equation}{\textup{#2(#1)#3}}
\crefrangeformat{equation}{\textup{#3(#1)#4--#5(#2)#6}}
\crefmultiformat{equation}{\textup{#2(#1)#3}}{ and \textup{#2(#1)#3}}
{, \textup{#2(#1)#3}}{, and \textup{#2(#1)#3}}
\crefrangemultiformat{equation}{\textup{#3(#1)#4--#5(#2)#6}}%
{ and \textup{#3(#1)#4--#5(#2)#6}}{, \textup{#3(#1)#4--#5(#2)#6}}%
{, and \textup{#3(#1)#4--#5(#2)#6}}

\Crefformat{equation}{#2Equation~\textup{(#1)}#3}
\Crefrangeformat{equation}{Equations~\textup{#3(#1)#4--#5(#2)#6}}
\Crefmultiformat{equation}{Equations~\textup{#2(#1)#3}}{ and \textup{#2(#1)#3}}
{, \textup{#2(#1)#3}}{, and \textup{#2(#1)#3}}
\Crefrangemultiformat{equation}{Equations~\textup{#3(#1)#4--#5(#2)#6}}%
{ and \textup{#3(#1)#4--#5(#2)#6}}{, \textup{#3(#1)#4--#5(#2)#6}}%
{, and \textup{#3(#1)#4--#5(#2)#6}}

\crefdefaultlabelformat{#2\textup{#1}#3}
\numberwithin{equation}{section}

\newtheorem{theorem} {Theorem}[section]
\newtheorem{proposition} {Proposition}[section]

\newtheorem{lemma}{Lemma}[section]

\newtheorem{counter-example}{counter-example}[section]
\newtheorem{remark} {Remark}[section]
\newtheorem{definition} {Definition}[section]
\newtheorem{question} {Question}[section]
\newtheorem{innercustomgeneric}{\customgenericname}
\providecommand{\customgenericname}{}
\newcommand{\newcustomtheorem}[2]{%
	\newenvironment{#1}[1]
	{%
		\renewcommand\customgenericname{#2}%
		\renewcommand\theinnercustomgeneric{##1}%
		\innercustomgeneric
	}
	{\endinnercustomgeneric}
}

\newcustomtheorem{customthm}{Theorem}
\newcustomtheorem{customlemma}{Lemma}

\def\CC{{\rm \kern.24em \vrule width.02em height1.4ex depth-.05ex \kern-.26emC}}

\def\TagOnRight
\def\AA{{it I} \hskip-3pt{\tt A}}

\def\QQ{\rlap {\raise 0.4ex \hbox{$\scriptscriptstyle |$}} {\hskip -0.1em Q}}

\makeatletter
\newcommand{\vo}{\vec{o}\@ifnextchar{^}{\,}{}}
\makeatother
\def\YYint#1#2#3{{\setbox0=\hbox{$#1{#2#3}{\iint}$}
		\vcenter{\hbox{$#2#3$}}\kern-.50\wd0}}
\def\XXint#1#2#3{{\setbox0=\hbox{$#1{#2#3}{\int}$}
		\vcenter{\hbox{$#2#3$}}\kern-.50\wd0}}

\makeatletter
\def\namedlabel#1#2{\begingroup
	\def\@currentlabel{#2}%
	\label{#1}\endgroup
}
\makeatother
\makeatletter
\newcommand{\rmh}[1]{\mathpalette{\raisem@th{#1}}}
\newcommand{\raisem@th}[3]{\hspace*{-1pt}\raisebox{#1}{$#2#3$}}
\makeatother

\newcommand{\descref}[2]{\hyperref[#1]{\textnormal{\textcolor{black}{}\textcolor{blue}{\bf #2}\textcolor{black}{}}}}
\newcommand{\dref}[2]{\hyperref[#1]{\textcolor{black}{(}\textcolor{blue}{\bf #2}\textcolor{black}{)}}}
\newcommand{\be} {\begin{eqnarray}}
	\newcommand{\ee} {\end{eqnarray}}
\newcommand{\Bea} {\begin{eqnarray*}}
	\newcommand{\Eea} {\end{eqnarray*}}
\newcommand{\rr}{\rightarrow}

\newcommand{\e}  {\epsilon}

\newcommand{\f}{\infty}

\newcommand{\R}{\mathbb{R}}

\DeclareMathOperator{\TV}{TV}

\DeclareMathOperator{\dy}{d}
\newcommand{\abs}[1]{\left| #1\right|}

\newcommand{\cd}{\mathcal{D}}

\newcommand{\RN}[1]{%
	\textup{\uppercase\expandafter{\romannumeral#1}}%
}

\newcounter{whitney}
\refstepcounter{whitney}

\newcounter{ineqcounter}
\refstepcounter{ineqcounter}
\makeatletter
\def\ps@pprintTitle{%
	\let\@oddhead\@empty
	\let\@evenhead\@empty
	\def\@oddfoot{}%
	\let\@evenfoot\@oddfoot}
\makeatother
\usepackage[doublespacing]{setspace}
\usepackage[titletoc,toc,page]{appendix}

\usepackage{pgf,tikz}
\usetikzlibrary{arrows}
\usetikzlibrary{decorations.pathreplacing}
\allowdisplaybreaks
\usepackage{cancel,soul,ulem}

\usepackage{etoolbox}
\usepackage{hyperref}

\usepackage{lipsum}% just for the example

\makeatletter
\def\@mkboth#1#2{}
\newlength\appendixwidth
\preto\appendix{\addtocontents{toc}{\protect\patchl@section}}
\newcommand{\patchl@section}{%
	\settowidth{\appendixwidth}{\textbf{Appendix }}%
	\addtolength{\appendixwidth}{1.5em}%
	\patchcmd{\l@section}{1.5em}{\appendixwidth}{}{\ddt}%
}
\makeatother

\begin{document}
	
	\begin{frontmatter}
		
		\title{{Existence and stability estimates for weak solutions of $p$-systems by front tracking scheme}
		}

		% \author[myaddress]{Adimurthi\tnoteref{thanksfirstauthor}}
		% \cortext[mycorrespondingauthor]{Corresponding author}
		% \ead{aditi@math.tifrbng.res.in and adiadimurthi@gmail.com}
		% \tnotetext[thanksfirstauthor]{Supported in part by Rajaramanna Fellowship.}
		% \author[myaddress1]{Piotr Gwiazda}
		% \ead{pgwiazda@mimuw.edu.pl}
		\author[myaddress]{Akash Parmar}
		\ead{a.parmar2@uw.edu.pl}
		% \author[myaddress]{Agnieszka Świerczewska-Gwiazda}
		% % \ead{aswiercz@mimuw.edu.pl}
	%\address[myaddress1]{Institute of Mathematics of Polish Academy of Sciences, Poland}
		\address[myaddress]{Faculty of Mathematics, Informatics, and Mechanics, University of Warsaw, Poland}
		
		%\date{\today}
		%\begin{center}{January 09, 2022} \end{center} 
		
		\begin{abstract}
We prove the existence of global weak entropy solutions to the $p$-system with the piecewise affine flux function. The construction is based on a wave-front tracking scheme for which we identify a necessary and sufficient condition for the occurrence of infinitely many fronts. For smooth fluxes, the theory is well established. However, the lower Lipschitz regularity of the flux requires a new method to obtain the interaction estimates. We also prove the stability of the solution in the sense of \cite{bia-col-02}, showing that solutions are Lipschitz continuous with respect to the $L^{\infty}$ norm of the difference of the flux derivatives. We obtain the convergence rate of the solutions from the piecewise affine flux to smooth flux by using the stability estimate.  
		\end{abstract}

	\end{frontmatter}
%	\begin{center}{January 09, 2022} \end{center} 
	%
	%
	\tableofcontents
	\section{Introduction}\label{introduction}
	The main aim of this article is to study the existence and stability results for the $p$-system by the means of the front tracking method. The p-system is one of the fundamental examples of strictly hyperbolic systems of conservation laws and  the $p$-system holds significant importance from an application perspective, as it naturally arises in the modeling of gas dynamics, traffic flow, fluid dynamics, and wave propagation \cite{LWR, traffic, smoller, dafermos2005hyperbolic}. Precisely, we discuss the following form of the $p$-system:
	 \begin{eqnarray}\label{p-sys}
	 \left\{\begin{array}{rlll}
	   v_{t}-u_{x}&=&0,  &x\in\R, t>0,\\
	    u_{t}+p(v)_{x}&=&0, &x\in\R, t>0,\\
	    (v(x,0), u(x,0))&=&(v_{0}(x), u_{0}(x)), &x\in\R,
	 \end{array}\right.
\end{eqnarray}
where $v>0, u$ and $p(v)=v^{-\gamma}$, for $\gamma>1$ are the specific volume, velocity and pressure of the fluid, respectively, and $(v_{0}, u_{0})\in BV\times BV$ is the initial data with $v_{0}>0$. 

We start with existing literature on the $p$-system \eqref{p-sys} to providing the background and context of our work. The existence and stability theory for system \eqref{p-sys} has been extensively studied in the framework of small and large BV solutions. Beginning with the pioneering work of Riemann \cite{riemann-60} who analyzed the Riemann problem for this system. For the existence of weak entropy solutions one can use several different method such as:  Glimm scheme \cite{Glimm}, wavefront tracking method \cite{diperna-1, bressan-front}, vanishing viscosity method \cite{bia-bre}, in the large BV framework \cite{nishida1, nishida2}. Slightly different model, Euler type flocking model related to the $p$-system, in \cite{amadori1, amadori2}, the authors established the existence of BV weak entropy solution by using the front tracking method.

In our setting, we consider the following piecewise affine approximation of system \eqref{p-sys}:
\begin{eqnarray}\label{appr-p}
	 \left\{\begin{array}{rlll}
	   v_{t}-u_{x}&=&0,  &x\in\R, t>0,\\
	    u_{t}+p_{n}(v)_{x}&=&0, &x\in\R, t>0,\\
	    (v(x,0), u(x,0))&=&(v_{0}(x), u_{0}(x)), &x\in\R,
	 \end{array}\right.
\end{eqnarray}
where $ p_{n}(v)$ is a piecewise affine approximation of pressure term $p(v)$, and $(v_{0}, u_{0})\in BV\times BV$ is the initial data with $v_{0}>0$. This approximation is inspired by the polygonal approximation of scalar conservation laws introduced by  Dafermos \cite{daf-poly} to construct wavefront tracking method, for further reference \cite[Section 2.3]{HoldenRisebro2002} .
The piecewise affine approximations  of the $p$-system are used in \cite{chen, lin-1, liu-xin}, to establish the lower bound on the density for $p$-system. In the context of Temple class systems piecewise affine approximation has been used in \cite{bian}. 

\subsection{Existence for the system \eqref{appr-p}.}\label{sec-exi}
For smooth pressure laws in \eqref{p-sys} the existence theory is well established. However, the pressure laws $p_{n}(v)$ in \eqref{appr-p} possess only the Lipschitz regularity. Consequently, the interaction estimate available in the smooth case, which needed to obtain uniform total variation bound for the solution, not applicable directly. Indeed, these estimates rely on the smoothness of the wave curves and typically comes from the application of the implicit function theorem. Therefore, the derivation of suitable interaction estimates for \eqref{appr-p} requires a different approach. In our work, we exploit the symmetric and explicit structure of the shock and rarefaction curves of the $p$-system to establish the required estimates. Infact, the obtained interaction estimats are third order estimates.  Our construction is based on the wavefront tracking method and a central technical difficulty is that, the wavefront tracking construction may generate infinitely many fronts within finite time.
\begin{question}
Does there a classification of the class of interaction which can generate the infinitely many interactions?
\end{question}
In this article, we provide a fine analysis of wavefront interaction and identify the necessary and sufficient criteria for the production of infinitely many fronts in the wavefront tracking approximation. More precisely, we identify the class of front interactions which can produce the infinitely many fronts. Infact, we provide an example of the explicit construction which shows that the infinitely many fronts are produced at some finite time $T$. For the necessary part, we construct a counting functional to prove that the interactions outside this class can interact only finitely many times provided the initially we start with finite fronts. The our example construction proves that we may indeed generate infinitely many fronts. Thus, to extend the front-tracking construction globally in time, we adopt \cite{bressan2000hyperbolic} Bressan's nonphysical fronts method. The key issue is to prove that the total strength of all nonphysical fronts can be made arbitrarily small. We obtain this estimate by introducing a generation weighted functional, which provides the control over the strength control of the high generation fronts. Thus, while the use of  nonphysical fronts is well known, the weighted functional provides a new mechanism for controlling their total strength. Moreover, weighted functional is independent of the $p$-system and can be used for any hyperbolic systems of conservation laws.

We emphasize an additional property of the our wavefront tracking construction. There is a fundamental difference between the front tracking approximation for the scalar and the system case. For scalar conservation laws with polygonal flux approximations, the front-tracking solution can be arranged so that its states remain among the nodal values of the approximating flux. In contrast, for systems this property is not generally preserved, since the Riemann solver may produce intermediate states outside the prescribed grid. In our construction, the piecewise affine structure of the pressure law makes the intermediate states to be computed explicitly in terms of the grid points of the approximation. This explicit representation is important from a numerical point of view, since it provides a computable description of the approximate Riemann solutions.

\subsection{Stability with respect to  perturbations of flux function}\label{sec-stab}
In this article, we also discuss the stability estimates with respect to perturbations of the flux functions. More precisely,  following the framework of \cite{bia-col-02}, we show that the $L^{1}$ norm between two standard Riemann semigroup, hereafter referrd to as SRS, is controlled by the $L^{\f}$  norm of the difference of the derivatives of the corresponding flux functions. Namely, let $u^f$ and $u^g$ be two solutions to systems of hyperbolic conservation laws with the same initial data but with different fluxes $f$ and $g$ respectively, 
$ u_t^f+f(u^f)_x=0 $
and
$ u_t^g+g(u^g)_x=0 $.
 Then the $L^1$ norm of the solutions can be estimated by
$$\|u^f(t)-u^g(t)\|_{L^1}\le C\|\nabla_uf(\cdot)-\nabla_ug(\cdot)\|_{\infty},$$
where C depends on time $t$ and BV norm of the solutions. We use these stability estimates to prove that the obtained weak solution to \eqref{appr-p} obtained from our construction converges to the SRS generated by the system \eqref{p-sys}. As a consequence, we obtain a convergence rate from solutions of the $p$-system with piecewise affine pressure laws to the corresponding solution associated with a smooth pressure law. Similar type of estimates also studied in~\cite{Chen_etal}. To provide further context to our stability result, we provide some state-of-the-art for stability of SRS in various frameworks. 

The existence and stability of the SRS for hyperbolic systems of conservation laws has been extensively studied. There are several results available for the existence of SRS: for the $2\times2$ systems \cite{bre-col}, for the general $n\times n$ systems \cite{bressan-cra}, for the $2\times2$ gas dynamics systems with large BV data \cite{col-rise}, for Temple class systems for nonconvex flux function \cite{bian}. For more references and discussions \cite{baiti, liu, bre-g}.

The stability of SRS received significant attention in the past few decades. Liu and Yang \cite{liu-stability} established the $L^{1}-$ stability of the SRS with respect to initial data for $2\times2$ hyperbolic systems of conservation laws. In \cite{bre-liu}, the authors extended $L^{1}-$stability for any $n\times n$ hyperbolic systems of conservation laws.  

Although not directly related to the present study, it is interesting to note from a broader perspective that, similar to scalar hyperbolic conservation laws, the degenerate parabolic equation
\cite{Vazquez}
\begin{equation}\label{deg-par}
    u_t=\Delta \phi(u),
\end{equation} 
  generates a non-expanding semigroup on $L^1$ space.
There are several attempts in the literature to prove the continuity of solutions with respect to $\phi$, \cite{Khazan1984, Benilan1981, Kalashnikov1978}, with \cite{Coc-Gri} being the most relevant among them. In this latter case the stability estimate has the following form
$$\|u^\phi(t)-u^\psi(t)\|_{L^1}\le C\sup_{s\in\R}|\sqrt{\phi'(s)}-\sqrt{\psi'(s)}|,$$
where $u^\phi$ and $u^\psi$ are solutions to \eqref{deg-par} with different nonlinearities $\phi$ and $\psi$ respectively. 
\subsection{Structure of the article}
The article is organized as follows: In Section \ref{sec-exi} and \ref{sec-stab} to provide some background and context we recall some existing results for existence and stability related to our work. In Section \ref{mainresults}, we state our main results with some required definitions, notations and remarks. In Section \ref{sec-scheme}, we provide some basic properties of the $p$-system and give the description of the wavefront tracking construction and fronts interaction analysis. In section \ref{app-RH},  we showed that the constructed solution is approximate weak solutions to \eqref{p-sys} and \eqref{appr-p}. In section \ref{sec-uni-bv}, we establish the main interaction estimate obtained for the piecewise affine pressure laws. Next Section \ref{uni-bv}, is dedicated to show that the obtained approximate solution satisfies the uniform BV bound for all time. Finally, Section \ref{main-estimates} provides the proof of the main existence and stability results. In \ref{inf-front-accu}, consists the explicit construction of the example which generates the infinitely many fronts in finite time, and uniform bound of the generation weighted functional.    
\section{Main results}\label{mainresults}
In this section, first we provide some definitions and notations which are required for the main results and then we state our results. We define 
\begin{align}\label{domain}
\mathcal{D}=cl\{U_{h}\in L^{1}(\mathbb{R};K\subset\mathbb{R}^{2}): U_{h} \mbox{ is a piecewise constant such that } TV(U_{h})<\delta\},
\end{align}
where $K$ is a compact set of $\R^{2}$ and $\delta>0$ is an sufficiently small constant.% One can see that $\mathcal{D}$ contains functions that have bounded total variation.
\begin{definition}[\cite{bressan2000hyperbolic}]\label{s6def1}\label{semigroup}
Let $T>0$ and $\cd\subset L^{1}(\mathbb{R};K\subset\R^{2})$ be a closed domain. A map $\mathcal{S}:\cd\times[0, T]\rr\cd$ is a standard Riemann semigroup generated by the \eqref{p-sys} if the following conditions hold:
\begin{enumerate}
	\item {\bf Semigroup property:} Every $U_{h}\in\cd$ and $t_{1},t_{2}\ge0$ satisfies
	\begin{align*}
	\mathcal{S}_{0}U_{h}=U_{h},~~ \mathcal{S}_{t_{1}}\mathcal{S}_{t_{2}}U_{h}=\mathcal{S}_{t_{1}+t_{2}}U_{h}.
	\end{align*}
	\item {\bf Lipschitz continuity:} For every $U_{h}, V_{h}$ and $t_{1},t_{2}\ge0$ there exists constants $L_{1}$ and $L_{2}$ such that 
	\begin{align*}
	||\mathcal{S}_{t_{1}}U_{h}-\mathcal{S}_{t_{2}}V_{h}||_{L^{1}}\le L_1||U_{h}-V_{h}||_{L^{1}}+L_2|t_{1}-t_{2}|.
	\end{align*}
	\item {\bf Consistency with Riemann solver:} The trajectory of map $(U_{h}, t)\mapsto\mathcal{S}_{t}U_{h}$ coincides with the solution of \eqref{p-sys} for piecewise constant initial data in $U_{h}\in\cd$ which obtained from the putting together the entropy solution of Riemann problem occurred from the jumps of $U_{h}$.     \end{enumerate}
\end{definition}
To make it self contained we give the definition of weak entropy solution to \eqref{p-sys}.
\begin{definition}[Weak solution]
We say $w:=(v,u)\in C([0,T], L^{1}_{loc}(\R);\R^{2})$ is a weak solution to \eqref{p-sys} if for all $\varphi\in C_{c}^{\f}(\R\times[0,T))$ it satisfies 
\begin{align*}
\int\limits_{0}^{T}\int\limits_{\R}(w\varphi_{t}+F(w)\varphi_{x})\dy x\dy t+\int\limits_{\R}w(x,0)\varphi(x,0)\dy x&=0,
\end{align*}
where $F(w)=(-u, p(v))$, and we call weak-entropy solution if for all non-negative $\varphi\in C_{c}^{\f}(\R\times[0,T))$ it satisfies entropy condition
\begin{align*}
\int\limits_{0}^{T}\int\limits_{\R}\eta(v,u)\varphi_{t}+Q(v,u)\varphi_{x}\dy x\dy t&\ge0,
\end{align*}
where $\eta$ is convex entropy and $Q$ is the corresponding entropy flux.
\end{definition}
\begin{definition}[Approximate front tracking weak solution]\label{app-fro-sol}
 We say $(v,u)\in C([0,T], L^{1}_{loc}(\R);\R^{2})$ is an approximate front tracking weak solution to \eqref{p-sys} if $(v,u)$ is a piecewise constant along finitely many discontinuities line in $xt$-plane and if the discontinuity line is shock/contact/rarefaction front then it agrees with the exact Riemann solution to system \eqref{appr-p}. If jump is a nonphysical front, then $v$-component of left and right state is same, and the speed of nonphysical front is zero.
 \begin{remark}
 We note that the rarefaction curve for system \eqref{appr-p} has affine structure, thus the exact solution of any Riemann problem for system \eqref{appr-p} always has piecewise constant structure. Definition \ref{app-fro-sol} has an immmediate consequence that the approximate front tracking weak solution satisfies the approximate Rankine-Hugoniot condition along every discontinuity which we establish in Subsection \ref{app-RH}. 
 \end{remark}
% \begin{align*}
% \int\limits_{0}^{T}\int\limits_{\R}(v\varphi_{t}-u\varphi_{x})\dy x\dy t+\int\limits_{\R}v(x,0)\varphi(x,0)\dy x&=0,\\
% \int\limits_{0}^{T}\int\limits_{\R}(u\varphi_{t}+p(v)\varphi_{x})\dy x\dy t+\int\limits_{\R}u(x,0)\varphi(x,0)\dy x&=0,
% \end{align*}
% and we call weak-entropy solution if for all non-negative $\varphi\in C_{c}^{\f}(\R\times[0,T))$ it satisfies entropy condition
% \begin{align*}
% \int\limits_{0}^{T}\int\limits_{\R}\eta(v,u)\varphi_{t}+Q(v,u)\varphi_{x}\dy x\dy t&\ge0,
% \end{align*}
%where $\eta$ is convex entropy and $Q$ is the corresponding entropy flux.
\end{definition}
Now, we state our main result.

\begin{theorem}[Weak solution to system \eqref{appr-p}]\label{exi-app-p} 
  For a fix $n\in\mathbb{N}$. Consider the system \eqref{appr-p} with the initial data $w_{0}(x,t):=(v_{0}(x,t), u_{0}(x,t))\in\mathcal{D}$. Then there exist global approximate wavefront tracking weak solution to \eqref{appr-p} which converges to global weak solution $w$ to \eqref{appr-p} corresponding to $w_{0}$.
  \end{theorem}

\begin{theorem}[Weak solution to system \eqref{p-sys}]\label{weak-p-syst}
Consider the system \eqref{p-sys} with the initial data $w_{0}(x,t)=(v_{0}(x,t),u_{0}(x,t))$ belongs to $\mathcal{D}$. Then there exists a global approximate front tracking weak solution $w^n$ to \eqref{p-sys} which converges to weak solution  $w$ to \eqref{p-sys} corresponding to $w_{0}$. Furthermore, $w$ forms a standard Riemann semigroup corresponding to $w_{0}$.
\end{theorem}

\begin{theorem}[Stability]\label{main-theorem}
Let  $S_{t}$ be the semigroup generated by the system \eqref{p-sys} corresponding to $w_{0}$.  For fix $n\in\mathbb{N}$, let $w^n(x,t)=(v^n(x,t), u^n(x,t))$ be the weak solution to \eqref{appr-p} corresponding to $w_{0}$. Then, there exists a constant $C$ such that  \begin{eqnarray}\label{df-df_n}
||S_{t}w_{0}-w^n(x,t)||_{L^{1}}\le C\cdot||DF-DF_{n}||_{L^{\f}} TV(w_{0})\,t.
\end{eqnarray}
Consequently,  
\begin{eqnarray}\label{df-dfn}
||S_{t}w_{0}-w^n(x,t)||_{L^{1}}\le \mathcal{O}\left(\frac{1}{2^{n}}\right)t.
\end{eqnarray}
\end{theorem}
% As a direct consequence of estimate \eqref{df-df_n} in Theorem \ref{main-theorem}, we obtain the following corollary.
% \begin{corollary}\label{corollary}
% Let $w(x,t)$ be a weak solution to  \eqref{p-sys}  and  $U^{n}(x,t)$ an approximate weak solution to \eqref{appr-p} obtained from our approximation scheme corresponding to same initial data  $w_{0}=(v_{0},u_{0})\in\mathcal{D}$. Then 
% \begin{eqnarray}
% ||w(x,t)-U^{n}(x,t)||_{L^{1}}\le\mathcal{O}\left(\frac{1}{2^{n}}\right)(1+t).
% \end{eqnarray}
% \end{corollary}
% Let us make the following remark about the discontinuities of the approximate solutions.
% \begin{remark}
% One can note that from our construction scheme the approximated weak solution $U^{n}(x,t)$ does not have the contact discontinuities, however, $S^{n}_{t}$ may have contact  discontinuities due to the piecewise linear approximation of the flux.  
% \end{remark}
Next we introduce the front tracking approximation scheme and  postpone the proof of the main results to Section~\ref{main-estimates}.

\section{Wave front  tracking scheme}\label{sec-scheme}
In this section, we present the front tracking approximation scheme to construct an approximate front tracking weak solution to \eqref{p-sys}. Let us start with defining approximation of pressure $p_{n}(v)$ in \eqref{appr-p}:
\begin{equation}\label{pn}
     p_{n}(v):=\frac{v-v_{i}^{n}}{v_{i+1}^{n}-v_{i}^{n}}p\left(v_{i+1}^{n}\right)+\frac{v_{i+1}^{n}-v}{v_{i+1}^{n}-v_{i}^{n}}p\left(v_{i}^{n}\right) \quad \mbox{ for } v\in\left[v_{i}^{n}, v_{i+1}^{n}\right].
\end{equation}
where $v_{i}^{n}\in\mathcal{V}=\{v_{i}^{n}: (v_{i+1}^{n}-v_{i}^{n})(p(v_{i}^{n})-p(v_{i+1}^{n}))=(v_{1}^{n}-v_{0}^{n})(p(v_{0}^{n})-p(v_{1}^{n}))\}$, for all $i\in\mathbb{Z}$, and we choose $v_{0}^{n}=1$ and $v_{1}^{n}=1+\frac{1}{2^{n}}$. %Choosing $v_{i}^{n}$ for $i>1$ and $i<0$, % we use the rarefaction curve corresponding to the flux $(u, p_{n}(v))$. Our choice of $v_{i}^{n}$ is such that whenever two rarefaction front interact the outgoing Riemann problem can be solved exactly for approximated system. More precisely, we need
% such that it satisfies the following relation
% \begin{eqnarray*}
% \int\limits_{v_{i-1}^{n}}^{v_{i}^{n}}\sqrt{-p_{n}'(y)}\dy y=\int\limits_{v_{i}^{n}}^{v_{i+1}^{n}}\sqrt{-p_{n}'(y)}\dy y,
% \end{eqnarray*}
% which can be rewritten as follows,
% \begin{eqnarray}\label{choice-equ}
% (v_{i}^{n}-v_{i-1}^{n})(p(v_{i-1}^{n})-p(v_{i}^{n}))=(v_{i+1}^{n}-v_{i}^{n})(p(v_{i}^{n})-p(v_{i+1}^{n})).
% \end{eqnarray}
%Let us make a very useful observation about relation \eqref{choice-equ},
% \begin{remark}
% One can note that due to our choice if $v_{0}^{n}$ and $v_{1}^{n}$ we can write,
% \begin{align*}
%   (v_{i}^{n}-v_{i-1}^{n})(p(v_{i-1}^{n})-p(v_{i}^{n}))&=\frac{(2^{n}+1)^{\beta}-2^{n\beta}}{(2^{n}+1)^{\beta}2^{n}}\\
%   \frac{(x_{i}-1)(x_{i}^{\beta}-1)}{x_{i}^\beta}&=\frac{(2^{n}+1)^{\beta}-2^{n\beta}}{(2^{n}+1)^{\beta}2^{n}}v_{i-1}^{\beta-1}
% \end{align*}
% where, $x_{i}:\frac{v_{i}}{v_{i-1}}$. By the strictly increasing property of the function on right hand side gives that $x_{i}\le x_{i+1}$
% \end{remark}
% From equation \eqref{choice-equ}, we can note,
%  \begin{eqnarray}\label{v_i-estimate} (v_{1}^{n}-v_{0}^{n})(p(v_{0}^{n})-p(v_{1}^{n}))=(v_{i+1}^{n}-v_{i}^{n})(p(v_{i}^{n})-p(v_{i+1}^{n})), \quad\mbox{ for any }i>1 \mbox{ or }i<0,
%  \end{eqnarray}
% by using the \eqref{v_i-estimate}, we can define the set $\mathcal{V}$ as
% \begin{equation*}
%     \mathcal{V}=\{v_{i}^{n}: (v_{1}^{n}-v_{0}^{n})(p(v_{0}^{n})-p(v_{1}^{n}))=(v_{i+1}^{n}-v_{i}^{n})(p(v_{i}^{n})-p(v_{i+1}^{n})), \quad\mbox{ for any }i>1 \mbox{ or }i<0\}.
% \end{equation*}
If we choose the initial data in some compact set does not contain 0, then for any $v_{i}^{n}$ and $v_{i+1}^{n}$ is in this compact set satisfies,
\begin{align*}
    \frac{\Tilde{C}}{2^{n}}=\Tilde{C}|v_{1}^{n}-v_{0}^{n}|\le|v_{i+1}^{n}-v_{i}^{n}|\le C|v_{1}^{n}-v_{0}^{n}|=\frac{C}{2^{n}}
\end{align*}
where constant $C$ and $\Tilde{C}$ are depends only on the compact set. %Let us give a notation $a_{i}=\frac{p(v_{0}^{n})-p(v_{1}^{n})}{p(v_{i}^{n})-p(v_{i+1}^{n})}>1$

\begin{center}
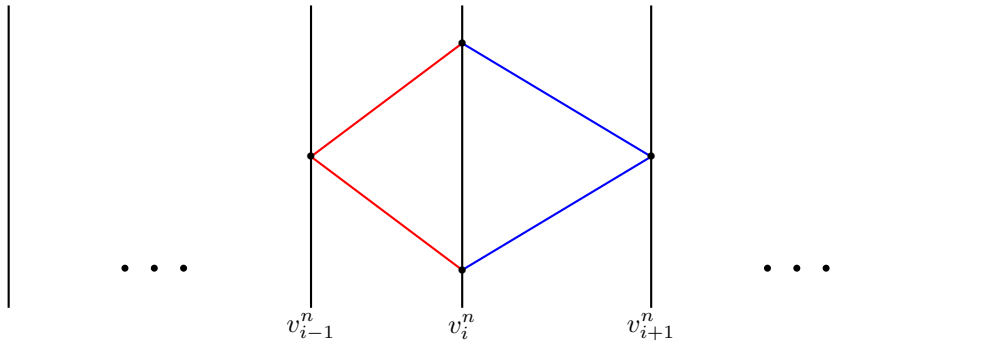

%vertical lines
    \begin{tikzpicture}
    \foreach \x in {-4, 0,2,4.5,9}
  \draw[thick] (\x,-.5) -- (\x,3.5);
    % \draw (-4,3.5)--(-4,-2);
    % \draw (0,3.5)--(0,-2);
    % \draw (2,3.5)--(2,-2);
    % \draw (4.5,3.5)--(4.5,-2);
    % \draw (9,3.5)--(9,-2);
    \draw[very thick] (6.5,0) node{\Huge $\cdots$};
    \draw[very thick] (-2,0) node{\Huge $\cdots$};
    
    %rarefaction curves
    \draw[red,thick] (0,1.5)--(2,0);
    \draw[red,thick] (0,1.5)--(2,3);
    \draw[blue,thick] (2,0)--(4.5,1.5);
    \draw[blue,thick] (2,3)--(4.5,1.5);
   % \draw[very thick] (4.5,0) node{$\cdots$};
    % \draw[very thick] (5.2,0) node{$\cdots$};
    
    %bullets
    \draw (0,1.5) node{\tiny $\bullet$};
    \draw (2,0) node{\tiny $\bullet$};
    \draw (2,3) node{\tiny $\bullet$};
    \draw (4.5,1.5) node{\tiny $\bullet$};

    %names&nodes
     \draw (0,-.75) node{\footnotesize $v_{i-1}^{n}$};
     \draw (2,-.75) node{\footnotesize $v_{i}^{n}$};
    \draw (4.5,-.75) node{\footnotesize $v_{i+1}^{n}$};
    
    %length lines
    % \draw [<->] (0,3.6)--(2,3.6);
    % \draw (1,3.9) node{\footnotesize $1/2^{n}$};
    % \draw [<->] (2,3.6)--(4.5,3.6);
    % \draw (3,3.9) node{\footnotesize $a_{1}/2^{n}$};
    \end{tikzpicture}
    \captionof{figure}{Illustration of Lax curves in $vu$-plane} \label{fig1}
\end{center}
\subsection{Shock and rarefaction curve for system \eqref{p-sys} and \eqref{appr-p}}
We consider the polygonal approximation of $p$-system \eqref{p-sys},
 \begin{eqnarray*}%\label{appr-p-2}
	 \left\{\begin{array}{rlll}
	   v_{t}-u_{x}&=&0,  &x\in\R, t>0,\\
	    u_{t}+p_{n}(v)_{x}&=&0, &x\in\R, t>0,\\
	    (v(x,0), u(x,0))&=&(\bar{v}_{0}^{n}(x), \bar{u}_{0}^{n}(x)), &x\in\R,
	 \end{array}\right.
\end{eqnarray*}
where $p_{n}(v)$ is a piecewise linear approximation of pressure term $p(v)$ as defined in \eqref{pn}, and $(\bar{v}_{0}^{n},\bar{u}_{0}^{n})$ is the piecewise constant approximation of initial data. %$(v_{0}(x), u_{0}(x))$ such that $v_{0}^{n}\in\left\{\frac{i}{2^{n}}:\ i\in\mathbb{N}\right\}$ for any fix $n\in\mathbb{N}$.
%We wish to have a well-defined approximating scheme to obtain an approximate entropy solution for \eqref{p-sys}. In order to get the approximating scheme for \eqref{p-sys} we use the Lax-curves of the system \eqref{appr-p}. First, we get the eigenvalues and the corresponding eigenvectors.
 Let $F_{n}(v,u):=(-u, p_{n}(v))$, and $DF_{n}(v,u)$  be the Jacobian matrix of $F_{n}(v,u)$ for $(v,u)\in\left(v_{i}^{n}, v_{i+1}^{n}\right)\times\R$. Moreover we define $DF_{n}\left(v_{i}^{n},u\right):=\lim_{v\rr v_{i}^{n}+}DF_{n}(v,u)$. Hence, system \eqref{appr-p} has two distinct eigenvalues
 \begin{align}\label{eig-val}
 \left\{\begin{array}{llll}
    \lambda_{1}(v,u)=&-\sqrt{\frac{p\left(v_{i}^{n}\right)-p\left(v_{i+1}^{n}\right)}{v_{i}^{n}-v_{i+1}^{n}}},  & (v,u)\in\left[v_{i}^{n}, v_{i+1}^{n}\right)\times\R, \\
     \lambda_{2}(v,u)=&\sqrt{\frac{p\left(v_{i}^{n}\right)-p\left(v_{i+1}^{n}\right)}{v_{i}^{n}-v_{i+1}^{n}}}, & (v,u)\in\left[v_{i}^{n}, v_{i+1}^{n}\right)\times\R,
 \end{array}\right.
 \end{align}
 we define $\Lambda_{i}^{n}:=\sqrt{\frac{p\left(v_{i}^{n}\right)-p\left(v_{i+1}^{n}\right)}{v_{i}^{n}-v_{i+1}^{n}}}$, so in future we use the notation  $\Lambda_{i}^n$ instead of writing the whole expression. The eigenvectors $r_{1}$ and $r_{2}$ corresponding to eigenvalues $\lambda_1(v,u)$ and $\lambda_2(v,u)$, respectively, are as follows,
 \begin{eqnarray}\label{eig-vec}
 \left\{\begin{array}{lllll}
    r_{1}(v,u)=&\left(1, \Lambda_{i}^{n}\right),  & (v,u)\in\left[v_{i}^{n}, v_{i+1}^{n}\right)\times\R, \\
     r_{2}(v,u)=&\left(1, -\Lambda_{i}^{n}\right),  & (v,u)\in\left[v_{i}^{n}, v_{i+1}^{n}\right)\times\R.
 \end{array}\right.
 \end{eqnarray}
%Therefore, from \eqref{eig-val} and \eqref{eig-vec} it follows that system \eqref{appr-p} is strictly hyperbolic. 
By using the Rankine-Hugoniot condition we can get the shock-curves $S_{1}^{n}$ and $S_{2}^{n}$,
\begin{equation}\label{S1}
    S_{1}^{n}: u=u_{l}-\sqrt{(v-v_{l})(p_{n}(v_{l})-p_{n}(v))}, \mbox{ for } v_{l}>v,
\end{equation}
\begin{equation}\label{s2}
    S_{2}^{n}: u=u_{l}-\sqrt{(v-v_{l})(p_{n}(v_{l})-p_{n}(v))}, \mbox{ for } v_{l}<v.
\end{equation}
Furthermore, the shock-curves $S_{1}^{n}$ and $S_{2}^{n}$ are Lipschitz curves in the $vu$-plane. Using the self-similar property, we obtain the rarefaction curves $R_{1}^{n}$ and $R_{2}^{n}$,
\begin{equation}\label{R1}
    R_{1}^{n}: u=u_{l}+\int\limits_{v_{l}}^{v}\sqrt{-p'_{n}(y)}\dy y \quad\mbox{ for } v>v_{l}.
\end{equation}
\begin{equation}\label{R2}
    R_{2}^{n}: u=u_{l}-\int\limits_{v_{l}}^{v}\sqrt{-p'_{n}(y)}\dy y \quad\mbox{ for } v<v_{l},
\end{equation}
From \eqref{R1}-\eqref{R2} we conclude that rarefaction curves $R_{1}^{n}$ and $R_{2}^{n}$ are also Lipschitz curves in $vu$-plane and consist of piecewise linear segments (see Fig \ref{fig1}). Moreover, the shock curves and rarefaction curves for system \eqref{p-sys} are given as follows
\begin{equation}\label{s1-p}
    S_{1}: u=u_{l}-\sqrt{(v-v_{l})(p(v_{l})-p(v))}, \mbox{ for } v_{l}>v,
\end{equation}
\begin{equation}\label{s2-p}
    S_{2}: u=u_{l}-\sqrt{(v-v_{l})(p(v_{l})-p(v))}, \mbox{ for } v_{l}<v,
\end{equation}
\begin{equation}\label{R1-p}
    R_{1}: u=u_{l}+\int\limits_{v_{l}}^{v}\sqrt{-p'(y)}\dy y \quad\mbox{ for } v>v_{l},
\end{equation}
\begin{equation}\label{R2-p}
    R_{2}: u=u_{l}-\int\limits_{v_{l}}^{v}\sqrt{-p'(y)}\dy y \quad\mbox{ for } v<v_{l}.
\end{equation}
\begin{remark}
A straightforward calculation shows that, in the $vu$-plane, curves $S_{i}^{n}$ and $R_{i}^{n}$ converges uniformly to $S_{i}$ and $R_{i}$, respectively, as $n$ goes to infinity. The curves $S_{i}^{n}$ coincide with $S_{i}$  and $R_{i}^{n}$ lies between $S_{i}$ and $R_{i}$ for $i=1, 2$, when $v, v_{l}\in\mathcal{V}$. 
\end{remark}
\subsection{Riemann solver}\label{riemann-solver}
We are now ready to introduce the Riemann solver to solve the Riemann problem obtained from the piecewise constant approximation of intial data.  Suppose there are $\{x_{k}\}_{k=1}^{N}$ discontinuity points at $\{t=0\}$ and for any fixed $m\in\mathbb{N}$, we choose piecewise constant approximation of initial data such that, 
\begin{eqnarray}\label{v-approx}
 \bar{v}_{0}^{m}\in\mathcal{V}, |v_{0}(x)-\bar{v}_{0}^{m}|\le\frac{1}{2^{m}} \quad \mbox{ and }\quad  \bar{u}_{0}^{n}\in\mathcal{U}=\left\{\frac{j}{2^{m}}: j\in\mathbb{Z}\right\}, |u_{0}(x)-\bar{u}_{0}^{m}|\le\frac{1}{2^{m}},
\end{eqnarray}
and $TV(v_{0}(x),u_{0}(x))\le TV(\bar{v}_{0}^{m}, \bar{u}_{0}^{m})$.
The existence of such an approximation is guaranteed by  \cite[Chapter 6]{bressan2000hyperbolic}. This approximation of initial data generates a collection of Riemann problems at the discontinuity points $\{x_{k}\}_{k=1}^{N}$.

Assume that at some discontinuity point $x_{k}$ on $\{t=0\}$, we have Riemann problem with the left and right states as $(v_{l}, u_{l})\in\mathcal{V}\times\mathcal{U}$ and $(v_{r}, u_{r})\in\mathcal{V}\times\mathcal{U}$, respectively. Using the shock and rarefaction curves, $S_{i}^{n}$ and $R_{i}^{n}$, defined in \eqref{S1}-\eqref{R2}, we solve the Riemann problem by connecting $(v_{l}, u_{l})$ to intermediate state $(v_{m}, u_{m})$ through 1-family curve and $(v_{m}, u_{m})$ connected to  $(v_{r}, u_{r})$ by a 2-family curve. However, the intermediate state $(v_{m}, u_{m})$ may not stay in $\mathcal{V}\times\mathcal{U}$. If the solution of Riemann problem consist a rarefaction wave then due to piecewise affine structure of rarefaction curve $R_{1}^{n}$ and $R_{2}^n$, it will divide into a rarefaction fan and each front in this fan propogate with speed $p'_{n}(v_{l})=p_{n}'(v_{i}^{n})$, $v_{l}\in\left[v_{i}^{n}, v_{i+1}^{n}\right)$ where $v_l$ is the left state $v$-component for corresponding front. The piecewise affine structure gives the piecewise constant solution to each Riemann problem. Additionally, one can verify that obtained piecewise constant solution satisfies the Rankine-Hugoniot condition for system \eqref{appr-p}, thus it is exact weak solution for \eqref{appr-p}. Furthermore, because of the affine structure of flux we can explicitly calculate the intermadiate state in terms of the grid points. Thus after each interaction we still have piecewise constant solution and we can continue this process unless at some finite time number of fronts become infinite. Therefore, we need to make sure that in our scheme the number fronts remains finite at all finite time. In next subsection, we discuss how we deal with this difficulty.

\subsection{Fronts counting analysis}
In this subsection, we point out exactly which type of fronts interaction can actually accumulate at some finite time $T$. In \ref{inf-front-accu}, we prove that for any given $\delta>0$ there exists an initial data $(v_0, u_0)\in BV\times BV$ such that $TV(v_0,u_0)\le\delta$, but the number of fronts still become infinite at some time $T<\f$. Moreover, this explicit construction of the example showed that the interaction pattern of this construction consist of infinitely many interaction between same family shock and rarefaction fronts. However, this is merely not a coincidence, we also prove that this is a necessary and sufficient criterion for the accumulation of interaction at some finite time, that is, any neighborhood of an accumulation point must contain infinitely many interactions between the same family shock and rarefaction fronts. In the next proposition, we will prove this criterion.

\begin{proposition}[Optimal interaction class for finite front]\label{finite-front}
Assume that, at time $t=0$, the front-tracking approximation starts from finitely many fronts. Further, assume that no same-family shock--rarefaction interaction occurs, that is, no interaction of the form
\[
S_i+R_i,\qquad i=1,2,
\]
takes place, where $S_i$ denotes a shock front of the $i$-th characteristic family and $R_i$ denotes a rarefaction front of the same family. Then the total number of interactions are finite. Consequently, the total number of fronts remains finite.
\end{proposition}

\begin{proposition}[Criterion for accumulation of interactions]\label{necessary-condition}
Let $z^*=(x^*,T^*)$ be accumulation point of interaction points for the front-tracking approximation. Then every neighbourhood of $z^*$ contains infinitely many same-family shock--rarefaction interactions.
\end{proposition}

\begin{remark}
Propositions~\ref{finite-front} and~\ref{necessary-condition} establish the
    necessary part of the criterion for the accumulation of interaction points.
    On the other hand, the example constructed in Section~\ref{inf-front-accu}
    shows that this criterion is also sufficient. Hence, together, these
    results give a complete characterization of the mechanism leading to the
    accumulation of interaction points.
\end{remark}

Before providing the proof of Proposition \ref{finite-front} and Proposition \ref{necessary-condition}, we need some definition and notations. Next, subsections are dedicated for the definitions and notation needed in the proof of the propositions.
\subsection{Definitions and notations needed for Proposition \ref{finite-front} and \ref{necessary-condition}}
In subsection \ref{uni-bv}, we establish the uniform total variation bound on the approximate weak solution. Since we already know that the approximate solution satisfies total variation bound for all time uniformly with respect to $n$. Thus front-tracking approximate solution remains in a fixed compact subset $K$ of $vu$-plane. In particular, the \(v\)-components of all states lie in a compact interval contained in \((0,\infty)\). Let
\begin{align*}
    Z=\{v_{i}^{n}\in\mathcal{V}: -N\le i\le N \}
\end{align*}
% \[
%         v_{-N}<v_{-N+1}<\cdots<v_{N-1}<v_{N}
% \]
be the finite set of grid points of the approximation \(p_n\) which intersect this compact \(v\)-range, and set $I_i:=[v_i^n,v_{i+1}^n],$ for $ i=-N,\ldots,N-1$. Let $\beta$ be a rarefaction front. If the left and right states of $\beta$ are $(v_l,u_l)$ and $(v_r,u_r)$, respectively, then we define the \(v\)-support of $\beta$ by
\[
        V(\beta):=
        \bigl[
        \min\{v_l,v_r\},
        \max\{v_l,v_r\}
        \bigr].
\]
From the piecewise affine structure of $p_n$, for every rarefaction front $\beta$ there exists an index
\(i\in\{0,\ldots,N-1\}\) such that $V(\beta)\subseteq I_i.$ At $t=0$, for each rarefaction front we assign a unique number as follows
\[
        \iota(\beta):=i
        \qquad\Longleftrightarrow\qquad
        V(\beta)\subseteq I_i .
\]
Next we define $ K(\beta):=N-1-\iota(\beta),$ then $
       0<K(\beta)\le K_{\max}:=2N-1$.
Thus $K(\beta)$ is the number of grid intervals strictly to the right of the grid interval containing \(\beta\). We denote the family of the front $\beta$ by $\mathscr F(\beta)$.
\subsection{Shock labels}
At time \(t=0\), we assign a distinct label to every initial shock front from an index set. Let \(\mathscr I_{s_1}\) and \(\mathscr I_{s_2}\) denote the finite sets of label indexes attached to the initial shocks of $1$-family and $2$-family, respectively. Set
\[
        S_{max}:=|\mathscr I_{s_1}|+|\mathscr I_{s_2}|,
\]
where $|\mathscr I_{s_1}|$ and $|\mathscr I_{s_2}|$ are the number of initial shock fronts of $1$-family and $2$-family, respectively. Every current shock front \(\gamma\) which is a continuation or merger of initial shocks carries a nonempty set of initial shock labels, denoted by $\mathscr L(\gamma)$. If two same-family shocks $\gamma_1$ and $\gamma_2$ interact and merge into an outgoing shock $\widetilde{\gamma}$, define
\[
        \mathscr L(\widetilde \gamma):=\mathscr L(\gamma_1)\cup\mathscr L(\gamma_2).
\]
If a shock crosses another front and continues as a shock, its label set is transported unchanged to the outgoing shock.
\subsection{Rarefaction labels}
We call a front is a rarefaction root if either it is an initial rarefaction front or a rarefaction front created at an interaction two same family shock fronts. Every rarefaction front root $\beta$ is assigned a unique label from an index set, denoted by $\mathscr I_{R}.$ A rarefaction front root $\beta$ interact with a front of other family, then every the outgoing rarefaction front of $\mathscr F(\beta)$ is called a ${\it branch }$ of $\beta$ and every branch carries the same label as its root label.  For a rarefaction front $\beta$, define the set of opposite shock labels by
\[
        \mathscr I_{\rm opp}(\beta):=
        \begin{cases}
        \mathscr I_{s_2}, & \text{ for } \mathscr F(\beta)=1,\\
        \mathscr I_{s_1}, & \text{ for }  \mathscr F(\beta)=2,
        \end{cases}
\]
and let $\mathscr H(\beta)$ is the label set of the shock fronts label of $\mathscr I_{\rm opp}(\beta)$ which has already crossed $\beta$. If \(\beta\) is a rarefaction root, set
\[
        \mathscr H(\beta):=\emptyset.
\]
If \(\beta\) interacts with an opposite-family rarefaction, then every outgoing rarefaction branch \(\beta'\) of \(\beta\) satisfies
\[
        \mathscr H(\beta')=\mathscr H(\beta).
\]
If \(\beta\) interacts with an opposite-family shock \(\gamma\), then every outgoing rarefaction descendant \(\beta'\) of \(\beta\) satisfies
\[
        \mathscr H(\beta')=\mathscr H(\beta)\cup\mathscr L(\gamma).
\]
We define a quantity $L(\beta)$ for a rarefaction front $\beta$ as follows 
\[
        L(\beta):=(K_{\max}+1)\nu(\beta)+K(\beta),
\] 
where $\nu(\beta):=\#\bigl(\mathscr I_{\rm opp}(\beta)\setminus \mathscr H(\beta)\bigr)$ and $0\le \nu(\beta)\le S_{\max}$.
\subsection{Counting polynomial}
 Let  \[ \mathcal A:=
        \{(\alpha,\beta): \alpha \text{ and } \beta \text{ are approaching fronts to each other}\}.\]
For an approaching pair \((\alpha,\beta)\), we define its degree by
\[
        d(\alpha,\beta):=
        \begin{cases}
        0, & \alpha,\beta \text{ are both shocks},\\
        L(\beta), & \alpha \text{ is shock, } \beta\text{ is a rarefaction },\\
        L(\alpha)+L(\beta), & \alpha, \beta \text{ are both rarefactions}.
        \end{cases}
\]
Since \(0\le L(\beta)\le L_{\max}\), we have
\[
        0\le d(A,B)\le D,
        \qquad
        D:=2L_{\max}.
\]
For \(d=0,\ldots,D\), set
\[
        \mathcal A_d:=
        \{(\alpha,\beta): \alpha \text{ and }\beta \text{ are approaching fronts such that }, d(\alpha,\beta)=d\},
\]
and
\[
        A_d:=\#\mathcal A_d.
\]
Finally, define the shock count
\[
        M_S:=
        \#\{\text{ all shock fronts }\}.
\]
The counting vector is
\[
        X:=
        \bigl(
        M_S,
        A_D,A_{D-1},\ldots,A_0
        \bigr)
        \in\mathbb N^{D+2}.
\]
Now, we define well defined order on the vector $G$. Given two vectors
\[
        X=
        \bigl(
        M_S,
        A_D,A_{D-1},\ldots,A_0
        \bigr)
\]
and
\[
        \widetilde X=
        \bigl(
        \widetilde M_S,
        \widetilde A_D,\widetilde A_{D-1},\ldots,\widetilde A_0
        \bigr),
\]
we write
\[
        X<\widetilde X
\]
if either
\[
        M_S<\widetilde M_S,
\]
or \(M_S=\widetilde M_S\) and there exists \(d_0\in\{0,\ldots,D\}\) such that
\[
        A_d=\widetilde A_d
        \qquad\text{for every }d>d_0,
\]
and
\[
        A_{d_0}<\widetilde A_{d_0}.
\]
Now in the next subsection we show that if interaction are belongs to interaction class to
\[
        S_2+S_1,\qquad
        S_2+R_1,\qquad
        R_2+S_1,\qquad
        R_2+R_1,\qquad
        S_1+S_1,\qquad
        S_2+S_2.
\]
then the counting polynomial monotonically decreasing at every interaction.
\subsection{Monotone decay of the counting functional}\label{decay-count-poly}
We now restrict the interaction class to
\[
        S_2+S_1,\qquad
        S_2+R_1,\qquad
        R_2+S_1,\qquad
        R_2+R_1,\qquad
        S_1+S_1,\qquad
        S_2+S_2.
\]
Same-family shock-rarefaction interactions \(S_i+R_i\), \(i=1,2\), are excluded. Let \(X^-\) and \(X^+\) denote the counting vector immediately before and immediately after an interaction, respectively. We prove that, for every interaction in the above restricted class,
\[
        X^+<X^- .
\]
\noindent
\textbf{1. Same-family shock-shock interactions.} Consider two shock same family shock fronts are interating, then two incoming shocks merge into one outgoing shock of the same family. Therefore the number of current shock fronts decreases by one:
\[
        M_S^+=M_S^- -1.
\]
The quantities \(A_d\) may increase, because new opposite-family rarefaction fronts may be created. However, since the first component \(M_S\) strictly decreases. Hence $ X^+<X^- .$

\medskip
\noindent
\textbf{2. Opposite-family shock-shock interactions.}
Consider an interaction of two different family shock fronts. The interacting pair is an approaching pair of degree \(0\). After the interaction, the two shocks are no longer approaching and no new front has been created. All other approaching pairs are preserved with the same degree. Hence, $ A_d^+=A_d^- , \text{ for }d=1,\ldots,D,$ and $ A_0^+=A_0^- -1.$ Therefore $ X^+<X^-.$

\medskip

\noindent
\textbf{3. Opposite-family shock-rarefaction interactions.}
Let \(\beta\) be the incoming rarefaction front and let \(\gamma\) be the incoming shock. The interacting pair is an approaching pair of degree $ L(\beta).$ After the interaction, the outgoing fronts are in separated order, and hence this interacting approaching pair disappears. Let \(\beta'\) be any outgoing rarefaction branch of \(\beta\). Then, from the definition of $\mathscr H(\beta'),$ we have 
\[
        \mathscr H(\beta')=\mathscr H(\beta)\cup\mathscr L(\gamma).
\]
Indeed, if \(\mathscr L(\gamma)\subseteq\mathscr H(\beta)\), then every label carried by \(\gamma\) has already crossed the rarefaction branch \(\beta\). After such a crossing, the shock label and the branch are in separating order. Hence after the interaction \(\gamma\) is not approaching to  \(\beta\). Therefore at least one new shock label is added to \(\mathscr H(\beta)\), and hence $ \nu(\beta')\le \nu(\beta)-1$. Since \(K(\beta')\le K_{\max}\), we obtain
\[
\begin{aligned}
        L(\beta')
        &=(K_{\max}+1)\nu(\beta')+K(\beta')  \\
        &\le (K_{\max}+1)(\nu(\beta)-1)+K_{\max} \\
        &< (K_{\max}+1)\nu(\beta)+K(\beta)
        =L(\beta).
\end{aligned}
\]
Thus,
\[
        L(\beta')<L(\beta).
\]
Therefore, any additional approaching pair after the interaction must have the strictly smaller degree. 
% Now let \(\delta\) be any front different from the two interacting fronts. If \(h_n\) and \(\delta\) formed an approaching pair before the interaction, then any new approaching pair formed by \(h_n'\) and \(\delta\), has strictly smaller degree. %Indeed, if \(\delta\) is a shock, then the degree changes from \(L(h_n)\) to \(L(h_n')\). If \(\delta\) is a rarefaction, then the degree changes from \(L(h_n)+L(\delta)\) to \(L(h_n')+L(\delta)\). In both cases the degree strictly decreases.
All approaching pairs not involving \(\beta\) are unchanged, except for the interacting pair itself, which disappears. Thus no coefficient \(A_d\) with degree larger than the highest degree affected by \(\beta\) can increase. At the highest degree where the counts change, the count strictly decreases. We obtain $ X^+<X^-.$
\medskip

\noindent
\textbf{4. Opposite-family rarefaction-rarefaction interactions.}
Let \(\beta_{1}\) be the incoming $2$-family rarefaction and let \(\beta_{2}\) be the incoming $1$-family rarefaction. The interacting pair is an approaching pair of degree
\[
        L(\beta_{1})+L(\beta_{2}).
\]
After the interaction, the outgoing branches of \(\beta_{i}\) for $i=1,2$ are in separated order. Hence no outgoing branch of \(\beta_{1}\) and outgoing branch of \(\beta_{2}\) form an approaching pair with each other. Therefore the interacting pair disappears and has no replacement between the two interacting rarefaction roots. Let \(\beta_{i}'\) be an outgoing branch of \(\beta_{i}\) for $i=1,2$. Since no shock label is crossed,
\[
        \mathscr H(\beta_{i}')=\mathscr H(\beta_{i}),
        \qquad
\nu(\beta_{i}')=\nu(\beta_{i}).
\]
If \(\iota(\beta_{i})=l_{i}\), then by the rarefaction-rarefaction interaction every outgoing rarefaction branch \(\beta_{i}'\) of \(\beta_{i}\) satisfies $\iota(\beta_{i}')\in\{l_{i},l_{i}+1\}$. Consequently,
\[
    K(\beta_{i}')\in\{K(\beta_{i}),K(\beta_{i})-1\},
\]
and therefore
\[
        L(\beta_{i}')\le L(\beta_{i}).
\]
Moreover, for each incoming rarefaction \(\beta_{i}\), at most one outgoing branch satisfies $L(\beta_{i}')=L(\beta_{i}).$
Indeed, equality can occur only for the unique branch, if it exists, whose \(v\)-support remains in the same grid interval as \(\beta_{i}\) and other branch lie in the next interval and have strictly smaller \(K\). Hence,
\[
        L(\beta_{i}')\le L(\beta_{i}),
\]
% and at most one branch of \(\rho\) has equality. Similarly, if \(\sigma'\) is an outgoing branch of \(\sigma\), then
% \[
%         L(\sigma')\le L(\sigma),
% \]
Now consider approaching pairs involving an outgoing branch and a third front \(\delta\). Equal-degree continuations can occur at most one-to-one with old approaching pairs, because each incoming rarefaction has at most one equal-level branch. All other new pairs have strictly smaller degree.

The old interacting pair \((\beta_{1},\beta_{2})\), of degree \(L(\beta_{1})+L(\beta_{2})\), disappears and has no equal-degree replacement. Hence no coefficient of degree larger than \(L(\beta_{1})+L(\beta_{2})\) increases, and at degree \(L(\beta_{1})+L(\beta_{2})\) the count strictly decreases. We get $ X^+<X^-.$
Combining the all these cases, every interaction in the restricted class satisfies $ X^+<X^-.$
Since \(X\in\mathbb N^{D+2}\) and the number of shock fronts are always non increasing for the restricted interaction. Hence, there cannot be infinitely many such strict decreases. Hence only finitely many interactions occur in the restricted class.

\begin{proof}[Proof of Proposition \ref{finite-front}]
The proof of the proposition immediately follows from the monotonic decay of vector $X$ which comes from the analysis of Subsection \ref{decay-count-poly}.
\end{proof}

\begin{proof}[Proof of Proposition \ref{necessary-condition}]
Assume that there exists a neighbourhood $B_{\e}$ of $z^*$
containing only finitely many same-family shock--rarefaction interactions. Let $I_{B_{\e}}$ be the finite set of all interaction points in $B_{\e}$ of the form
\[
        S_i+R_i,\qquad i=1,2,
\]
where $S_{i}$ and $R_{i}$ are shock and rarefaction front of the \(i\)-th family, respectively. Since
$I_{B_{\e}}$ is finite, we can choose a smaller neighbourhood \(B_{\e'}\subset B_{\e}\) of
\(z^*\) such that
\[
        B_{\e'}\cap I_{B_{\e}}=\emptyset .
\]
Thus no same-family shock--rarefaction interaction occurs inside $B_{\e'}$. Since the approximate solution has uniformly bounded total variation and takes
values in a compact subset of the state space, all front speeds are uniformly
bounded. Hence there exists $\lambda_{max}>0$ such that every front speed
\(\lambda\) satisfies $ |\lambda|\leq \lambda_{max}$.
Since $B_{\e'}$ is a neighbourhood of $z^*$, we may choose \(h_{1}>0\) and
\(t_0<T^*\), with \(t_0\) sufficiently close to \(T^*\), such that the backward
cone
\[
        \Delta(z^*)
        :=
        \left\{
        (t,x):\ t_0<t<T^*,\ 
        |x-x^*|<h_{1}+\lambda_{max}(T^*-t)
        \right\}
\]
is contained in $B_{\e'}$. By construction, \(\Delta(z^*)\subset B_{\e'}\), and no same-family shock--rarefaction
interaction occurs inside $B_{\e'}$. Consequently, every interaction inside $\Delta(z^*)$ belongs to the restricted
interaction class. The cone $\Delta(z^*)$ starts with only finitely many fronts
entering through the bottom.
Therefore, by Proposition \ref{finite-front} only
finitely many interaction points can occur inside $\Delta(z^*)$. This contradicts the fact that \(z^*\) is an accumulation point. Therefore, every
neighbourhood of \(z^*\) contains infinitely many same-family shock--rarefaction
interactions. This establish the necessary part of the criterion for the accumulation of interaction points. If there is interaction of same family shock-rarefaction fronts, that is of the form $S_{i}+R_{i}$, then example in \ref{inf-front-accu} implies that by choosing the $n$ large enough, we can get same interaction pattern as in example which generate an accumulation of interaction points.
\end{proof}
\subsection{Adaption of nonphysical shock method for Riemann solver}
Now, due to the construction of example of infinite number of fronts can accumulate at some finite time $T$, we know that we can not construct a global front tracking approximation by using the exact Riemann solution for system \eqref{appr-p}. Thus, to resolve this problem we adapt the Bressan's nonphysical shock method. However, our previous analysis suggest that we need to modify the Riemann solver only for the Riemann problem generated by the interaction of same family shock-rarefaction interaction.
\subsection{Modified Riemann solver}
Assume that $\beta$ and $\gamma$ are the two incoming shock and rarefaction fronts of $i$-th family, respectively.  Let
\[
        \beta:\ (v_{l}, u_{l})\to (v_m,u_m),\qquad
        \gamma:\ (v_m,u_m)\to (v_r,u_r),
\]
hence,
\[\beta: u_{m}-u_{l}=(-1)^{3-i}\sqrt{(p_{n}(v_{m})-p_{n}(v_{l}))(v_{l}-v_{r})},\qquad
\gamma:u_{r}-u_{m}=(-1)^{3-i}\int\limits_{v_{m}}^{v_{r}}\sqrt{-p'_{n}(y)}\dy y,
\]
now we can define a non physical right state as follows
\begin{align*}
    (\hat{v}_r, \hat{u}_{r})=(v_{r}, u_{l}-(-1)^{3-i}\sqrt{(p_{n}(v_{m})-p_{n}(v_{l}))(v_{l}-v_{r})}),
    \end{align*}
finally the Riemann problem generated by the interaction of $\beta$ and $\gamma$ has modified solution as follows,
\[
        \beta':\ (v_{l}, u_{l})\to (\hat{v}_r,\hat{u}_r),\qquad
        \Gamma_{np}:\ (\hat{v}_r,\hat{u}_r)\to (v_r,u_r),
\]
where $\beta'$ is the outgoing shock front of $i$-th family and $\Gamma_{np}$ is the nonphysical front propogating with zero speed. The strength of nonphysical front is defined by $|\mu_{np}|=|\hat{u}_{r}-u_{r}|$.
 
 Now assume nonphysical front $\Gamma_{np}$ interacting with front $\beta$ of $i$-th family. By the definition of nonphysical front the $v$-component of left and right states of nonphysical fronts are same. Let 
\[
        \Gamma_{np}:\ (v_{l}, u_{l})\to (v_l,u_m),\qquad
        \beta:\ (v_l,u_m)\to (v_r,u_r),
\]
after the interaction let $\beta'$ is the outgoing front of $i$-th family and $\Gamma_{np}'$ is the outgoing nonphysical front, then we define
\[
        \beta':\ (v_{l}, u_{l})\to (\hat{v}_r,\hat{u}_r),\qquad
        \Gamma_{np}:\ (\hat{v}_r,\hat{u}_r)\to (v_r,u_r),
\]
where $(\hat{v}_r, \hat{u}_{r})=(v_{r}, u_{l}-(-1)^{3-i}\sqrt{(p_{n}(v_{m})-p_{n}(v_{l}))(v_{l}-v_{r})})$. For a fix $n\in\mathbb{N}$, we define a parameter $h_{n}>0$. Assume $\beta$ and $\gamma$ are the two incoming shock and rarefaction fronts of same family with strength $|\mu_{\beta}|$ and $|\mu_{\gamma}|$, respectively. We use the modified Riemann solver only when the $|\mu_{\beta}||\mu_{\gamma}|\le h_{n}$. In addition, $|\hat{u}_{r}-u_{r}|=|u_{l}-u_{m}|$, follows directly from the definition.

\section{Approximate solution to \eqref{p-sys}}\label{app-RH}
Next, we will show that the front tracking approximate solution obtained from our construction is generate an approximate weak solution to system \eqref{p-sys} in the sense that it satisfies the approximate Rankine Hugoniot condition, i.e.,
\begin{eqnarray}\label{weak-sol}
|\mathcal{E}_{n}|=|(\mathcal{E}^{1}_{n},\mathcal{E}^{2}_{n})|:=|(F(v_{r},u_{r})-F(v_{l},u_{l}))-\xi((v_{r},u_{r})-(v_{l},u_{l}))|\le\frac{C}{2^n}|v_{r}-v_{l}|,
\end{eqnarray}
 and  let $(Q,\eta)$ be a convex flux-entropy pair. Then 
\begin{equation}\label{e-sol}
    \mathcal{Q}:=(Q(v_{r},u_{r})-Q(v_{l},u_{l}))-\xi(\eta(v_{r},u_{r})-\eta(v_{l},u_{l}))\le\frac{C}{2^n}|v_{r}-v_{l}|,
\end{equation}
 where $\xi$ is the speed of the front between $(v_{l},u_{l})$ to $(v_{r},u_{r})$. By the construction itself the solution of each Riemann problem is an exact weak solution to \eqref{appr-p} except for the Riemann problem generated by the interaction of same family shock-rarefaction interactions. The front tracking approximation obtained from construction can have rarefaction fronts, shock fronts, contact discontinuities, and nonphysical fronts. We can prove the estimate \eqref{weak-sol}-\eqref{e-sol} for each of theses  except for the nonphysical front case which satisfies $|\mathcal{E}^{1}_{n}|=|u_{l}-u_{r}|$ and
          $|\mathcal{E}^{2}_{n}|=0$.
 \begin{itemize}
     \item [1.] Consider the case of 1-family rarefaction front is between  $(v_{l},u_{l})$ and  $(v_{r},u_{r})$. For each 1-family rarefaction front we know $v_{l}<v_{r}$, and $v_{l},v_{r}\in[v_{i}^{n}, v_{i+1}^{n}]$ for some $v_{i}^{n}\in\mathcal{V}$ and front speed $\xi=-\sqrt{-p'(v_{i}^{n})}$. From \eqref{R1}, we have $u_{l}-u_{r}=-\int\limits_{v_{l}}^{v_{r}}\sqrt{-p'(v_{i}^{n})}$. Therefore, 
     \begin{align*}
          |\mathcal{E}^{1}_{n}|&=|(u_{l}-u_{r})+\sqrt{-p'(v_{i}^{n})}(v_{r}-v_{l})|=0\\
          |\mathcal{E}^{2}_{n}|&=|p(v_{r})-p(v_{l})+\sqrt{-p'(v_{i}^{n})}(u_{r}-u_{l})|=\left|\int\limits_{v_{l}}^{v_{r}}p'(y)-p'(v_{i}^{n})\dy y\right|\le \frac{C}{2^{n}}(v_{r}-v_{l})%\left|\int\limits_{v_{l}}^{v_{r}}\left(\sqrt{-p'(v_{i}^{n})}-\sqrt{-p'(y)}\right)\dy y\right|\\
        %  &=\left|\int\limits_{v_{l}}^{v_{r}}\frac{p'(y)-p'(v_{i}^{n})}{\sqrt{-p'(v_{i}^{n})}+\sqrt{-p'(y)}}\dy y\right|\le \frac{C}{2^{n}}(v_{r}-v_{l}).
     \end{align*}
     Hence, $|\mathcal{E}_{n}|\le \frac{C}{2^{n}}(v_{r}-v_{l})$. For 2-family rarefaction front, we get the estimate in similar fashion.
     \item[2.] Consider the case of 1-family shock front is between  $(v_{l},u_{l})$ and  $(v_{r},u_{r})$. For each 1-family shock curve, we know $v_{l}>v_{r}$, and from \eqref{s1}, we have $\xi=\frac{u_{l}-u_{r}}{v_{r}-v_{l}}$. Hence,
     \begin{align*}
          |\mathcal{E}^{1}_{n}|&=|(u_{l}-u_{r})-\frac{u_{l}-u_{r}}{v_{r}-v_{l}}(v_{r}-v_{l})|=0\\
          |\mathcal{E}^{2}_{n}|&=|p(v_{r})-p(v_{l})-\frac{u_{l}-u_{r}}{v_{r}-v_{l}}(u_{r}-u_{l})|=\left|p_{n}(v_{l})-p(v_{l})+p(v_{r})-p_{n}(v_{r})\right|,
     \end{align*}
    if $v_{l}, v_{r}\in\mathcal{V}$, then $|\mathcal{E}^{2}_{n}|=0$, otherwise, suppose $v_{l}\in[v_{i}^{n},v_{i+1}^{n}]$ and $v_{r}\in[v_{k}^{n},v_{k+1}^{n}]$, then
    \begin{align*}
        \left|p_{n}(v_{l})-p(v_{l})+p(v_{r})-p_{n}(v_{r})\right|&=\left|\int\limits_{v_{i}^{n}}^{v_{l}}p'(v_{i}^{n})-p'(y)\dy y\right|+\left|\int\limits^{v_{k+1}^{n}}_{v_{r}}p'(y)-p'(v_{k}^{n})\dy y\right|\\&\le\frac{C}{2^n}(v_{l}-v_{r}).
    \end{align*}
    \item[3.] In the case of contact discontinuities, we must have $v_{l},v_{r}\in[v_{i}^{n}, v_{i+1}^{n}]$ for some $v_{i}^{n}\in\mathcal{V}$, hence this case can be handled similar to the case of rarefaction front.
    \item[4.] Let $(v_{l}, u_{l})$ and $(v_{r},u_{r})$ connected by a nonphysical front. By the definition of nonphysical front  $\xi=0$ and $v_{l}=v_{r}$. Then, $|\mathcal{E}^{1}_{n}|=|u_{l}-u_{r}|$ and
          $|\mathcal{E}^{2}_{n}|=0$

 \end{itemize}
 The estimate \eqref{e-sol} for approximate entropy inequality can be done completely similar fashion, hence we omitted the details to avoid the repetition of the calculations. Next we show the total strength of nonphysical fronts can be made arbitrary small.
\subsection{Total strength of nonphysical fronts}
In this section we will show that the total strength of nonphysical fronts is smaller than any given $\varepsilon>0$. Before starting the proof we recall the definition of {\it ``front generation''} from \cite{bressan2000hyperbolic}.
 \begin{definition}
 The generation of all fronts at time $t=0$ are defined $g=1$. Furthermore, suppose two fronts $\beta_{1}$ and $\beta_{2}$  with generations $g_{1}$ and $g_{2}$, respectively, are interacting then generation of outgoing fronts are defined as follows,
 \begin{enumerate}
      \item If $\beta_{1}$ and $\beta_{2}$ are fronts from the different family. Then, generation of outgoing fronts remain same generation as the generations of incoming fronts of the corresponding families

     \item If $\beta_{1}$ and $\beta_{2}$ are fronts from the same family. Then
    \begin{enumerate}
        \item Generation of outgoing fronts from the same family as $\beta_{i}$ are defined as $\min\{g_{1}, g_{2}\}$.
         \item Generation of outgoing fronts of different family from $\beta_{i}$ are defined as $\max\{g_{1}, g_{2}\}+1$.
    \end{enumerate}
     \end{enumerate}
 \end{definition}

\begin{lemma}[Smallness of high-generation nonphysical fronts]
\label{nonphysical-smallness}
 Let $\chi>1$ is fixed. Let $\mathcal{F}_{np}(t)$ denote the set of all nonphysical fronts at time $t$, and $g(\Gamma_{np})$ is the generation of front $\Gamma_{np}\in\mathcal{F}_{np}(t)$. Then, for every integer \(k\geq 1\) and every \(t\geq 0\),
\[
        \sum_{\substack{\Gamma\in\mathcal{F}_{np}(t)\\ g(\Gamma_{np})>k}}
        |\mu_{\Gamma_{np}}(t)|
        \leq
        C\,\chi^{-(k+1)}.
\]
Consequently, for any $\varepsilon>0$ given, by choosing the $k$ large enough 
\[
        \sum_{\substack{\Gamma\in\mathcal{F}_{np}(t)\\ g(\Gamma_{np})>k}}
        |\mu_{\Gamma_{np}}(t)|
        \leq\varepsilon.
\]
Thus, the total strength of high-generation nonphysical fronts can be
made arbitrarily small by choosing \(k\) sufficiently large.
\end{lemma}
\begin{proof} 
Let $\mathcal F(t)$ denote the set of all fronts at time $t$, and $g(\gamma)$ is the generation of front $\gamma$. Define 
\[ V_{\chi}(t) := \sum_{\gamma\in\mathcal F(t)}|\mu_\gamma(t)|\,\chi^{g(\gamma)}. \]
From Lemma \ref{wei-gen-bound}, we know that $V_{\chi}(t)$ is uniformly bounded, i.e., $V_{\chi}(t)\le C$. 
Fix \(k\geq 1\) and \(t\geq 0\). Let \(\Gamma_{np}\in\mathcal{F}_{np}(t)\) be a
nonphysical front with $ g(\Gamma_{np})>k.$ Hence
\[
        |\mu_{\Gamma_{np}(t)}|
        \leq
        \chi^{-(k+1)}
        |\mu_{\Gamma_{np}(t)}|\,\chi^{g(\Gamma_{np})}.
\]
Summing over all nonphysical fronts with generation strictly larger than \(k\),
we obtain
\[
        \sum_{\substack{\Gamma\in\mathcal{F}_{np}(t)\\ g(\Gamma)>k}}
        |\mu_{\Gamma_{np}(t)}|
        \leq
        \chi^{-(k+1)}
        \sum_{\substack{\Gamma\in\mathcal{F}_{np}(t)\\ g(\Gamma)>k}}
        |\mu_{\Gamma_{np}(t)}|\,\chi^{g(\Gamma_{np})}.
\]
Since \(\mathcal{F}_{np}(t)\subset \mathcal F(t)\), thus
\[
        \sum_{\substack{\Gamma\in\mathcal{F}_{np}(t)\\ g(\Gamma)>k}}
        |\mu_{\Gamma_{np}(t)}|
        \leq
        \chi^{-(k+1)}V_{\chi}(t)\le C\chi^{-(k+1)}.
\]
Since $\chi>1$, therefore, for every \(\varepsilon>0\), there exists \(k\) sufficiently large
such that 
$ C\chi^{-(k+1)}<\varepsilon.$ Hence,
\[
        \sum_{\substack{\Gamma\in\mathcal{F}_{np}(t)\\ g(\Gamma)>k}}
        |\mu_{\Gamma_{np}(t)}|
        \leq\varepsilon.
\]
This completes the proof of the lemma.
\end{proof}

\begin{lemma}[Fixed-generation counting and low-generation smallness]
\label{lem:fixed-generation-counting-low-smallness}
Let $\mathcal{F}_{np}(t)$ denote the set of all nonphysical fronts at time $t$, and $g(\Gamma_{np})$ is the generation of front $\Gamma_{np}\in\mathcal{F}_{np}(t)$.
Then, for every $k\geq 1$, there exists a constant $P_{k}>0$, such that
\begin{align*}
\sum_{\substack{\Gamma\in\mathcal{F}_{np}(t)\\ g(\Gamma_{np})<k}}
        |\mu_{\Gamma_{np}}(t)|
        \leq
        P_{k}h,  
\end{align*}
where $h$ is the threshold upper bound of the nonphysical front. Consequently, for any $\varepsilon>0$ given, by choosing the $h$ large enough 
\[
        \sum_{\substack{\Gamma\in\mathcal{F}_{np}(t)\\ g(\Gamma_{np})<k}}
        |\mu_{\Gamma_{np}}(t)|
        \leq\varepsilon.
\]
\end{lemma}
\begin{proof}
For fix $n$, let the initial datum be piecewise constant and assume that it has
only finitely many initial fronts. Let \(N_0\) denote the number of initial
fronts. We first prove the fixed-generation counting estimate. Since \(n\) is fixed and all states remain in a fixed compact set, there is a
constant $M_n$, independent of $h$, such that each Riemann problem in
the construction produces at most \(M_n\) outgoing fronts. This includes the
fronts in rarefaction fans and the possible nonphysical front introduced by
the simplified solver. We first prove the estimate for \(k=1\). Since initially we have $N_0$ fronts of generation $1$. The number of possible interactions between generation $1$ fronts is bounded
by $N_0^2 .$ At each such interaction, at most \(M_n\) outgoing fronts can be produced.
Therefore the total number of generation $1$ fronts are
bounded by $N_0+M_nN_0^2$.
Hence
\begin{align*}
\mathcal{F}_{np}(1, t):=
\#\{\Gamma_{np}\in\mathcal{F}(t): g(\Gamma_{np})\leq 1\}\le P_{1}:=N_0+M_nN_0^2.
\end{align*}
By induction hypothesis assume that for some $k\geq1$, there exists a
constant \(P_{k}\), independent of parameter $h$, such that
\begin{align*}
\mathcal{F}_{np}(k, t):=
\#\{\Gamma_{np}\in\mathcal{F}(t): g(\Gamma_{np})\leq k\}\le P_{k}.
\end{align*}
 We prove the estimate for generation $k+1$. In order to create a front of generation $k+1$, both incoming fronts must have generation at most $k$, and at least one of them must have generation exactly $k$. By the induction
hypothesis, there are at most \(P_{k}\) fronts of generation $k$.
Consequently, the number of interactions between fronts of generation at most
\(k\) is bounded by $ P_{k}^2$. At each such interaction, at most $M_n$ new fronts can be generated.
Therefore the number of fronts of generation at most $k+1$ is bounded by
\begin{align*}
\mathcal{F}_{np}(k+1, t):=
\#\{\Gamma_{np}\in\mathcal{F}(t): g(\Gamma_{np})\leq k+1\}\le P_{k+1}:=P_{k}+M_nP_{k}^2.
\end{align*}
This constant depends on \(k\), \(n\), and the number of initial fronts, but
it is independent of $h$. By induction, the estimate holds for all $k\in\mathbb{N}$. We now prove the smallness of low-generation nonphysical fronts. Since each nonphysical front $|\mu_{\Gamma_{np}}|\le h$, hence
\begin{align*}
\sum_{\substack{\Gamma\in\mathcal{F}_{np}(t)\\ g(\Gamma_{np})<k}}
        |\mu_{\Gamma_{np}}(t)|
        \leq
        P_{k}h,  
\end{align*}
Since $P_{k}$ is independent of $h>0$. Hence, for given $\varepsilon>0$ we can choose $h>0$ sufficiently
small such that
\begin{align*}
\sum_{\substack{\Gamma\in\mathcal{F}_{np}(t)\\ g(\Gamma_{np})<k}}
        |\mu_{\Gamma_{np}}(t)|
        \leq
        \varepsilon. 
\end{align*}
This completes the proof.
\end{proof} 
\section{Preliminaries and interaction estimates}\label{sec-uni-bv}
In this section we prove some important technical lemmas which we require for the interaction estimates.
Next lemma provides a qualitative comparison between the rarefaction and shock curves of the system \eqref{appr-p}.
\subsection{Preliminaries lemmas required for the interaction estimates}
\begin{lemma}\label{shock-bound}
Let $S_i$ be a shock curve such that $(v_{l}, u_{l})$ and $(v_{r}, u_{r})$ are left and right states, respectively, with $v_{l}\in[v_{j},v_{j+1}]$ and $v_{r}\in[v_{k},v_{k+1}]$. Then,
\begin{enumerate}
    \item For $S_{1}$, $v_{l}>v_{r}$, then
    \begin{enumerate}
        \item $|u_{l}-u_{r}|\ge(v_{l}-v_{j})\Lambda_{j}+(v_{k+1}-v_{r})\Lambda_{k}+\sum\limits_{k+1}^{j-1}(v_{z+1}-v_{z})\Lambda_{z}$,
      \item\label{Lambda_bound} $(v_{l}-v_{r})\Lambda_{j+1}\le|u_{l}-u_{r}|\le(v_{l}-v_{r})\Lambda_{k+1}$.
    \end{enumerate}
    \item For $S_{2}$, $v_{l}<v_{r}$, then
  \begin{enumerate}
        \item $|u_{l}-u_{r}|\ge(v_{r}-v_{k})\Lambda_{k}+(v_{j+1}-v_{l})\Lambda_{j}+\sum\limits_{j+1}^{k-1}(v_{z+1}-v_{z})\Lambda_{z}$,
      \item $ (v_{l}-v_{r})\Lambda_{k+1}\le|u_{l}-u_{r}|\le(v_{l}-v_{r})\Lambda_{j+1}$,
    \end{enumerate}
\end{enumerate}
where $\Lambda_{z}=\sqrt{\frac{p(v_{z})-p(v_{z+1})}{v_{z+1}-v_{z}}}$.
\end{lemma}
\begin{proof}
We prove the estimate for the $S_{1}$, and the estimate for $S_{2}$ can be done in similar fashion. by using the expression of the shock curve $S_{1}$, we can write,
% \begin{align*}
% |u_{l}-u_{r}|&=\sqrt{(v_{l}-v_{r})(p(v_{r})-p(v_{l}))}\\
% &= \sqrt{\left((v_{l}-v_{j})+(v_{k+1}-v_{r})+\sum\limits_{k+1}^{j-1}(v_{z+1}-v_{z})\right)\left((v_{l}-v_{j})\Lambda_{j}^{2}+(v_{k+1}-v_{r})\Lambda_{k}^{2}+\sum\limits_{k+1}^{j-1}(v_{z+1}-v_{z})\Lambda_{z}^{2}\right)}\\
% &=\sqrt{\left((v_{l}-v_{j})\Lambda_{j}\right)^{2}+\left((v_{k+1}-v_{r}\right)\Lambda_{k})^{2}+\sum\limits_{k+1}^{j-1}\left((v_{z+1}-v_{z})\Lambda_{z}\right)^{2}+}
% \end{align*}
\begin{align*}
(u_{l}-u_{r})^2
&= (v_{l}-v_{r})\bigl(p_{n}(v_{r})-p_{n}(v_{l})\bigr) \\[0.5ex]
&= \scalebox{0.93}{
$\Bigl((v_{l}-v_{j})+(v_{k+1}-v_{r})
+\sum\limits_{k+1}^{j-1}(v_{z+1}-v_{z})\Bigr)
\Bigl((v_{l}-v_{j})\Lambda_{j}^{2}
+(v_{k+1}-v_{r})\Lambda_{k}^{2}
+\sum\limits_{k+1}^{j-1}(v_{z+1}-v_{z})\Lambda_{z}^{2}\Bigr)$
} \\[0.5ex]
&= \bigl((v_{l}-v_{j})\Lambda_{j}\bigr)^{2}
+\bigl((v_{k+1}-v_{r})\Lambda_{k}\bigr)^{2}+(v_{l}-v_{j})(v_{k+1}-v_{r})\left(\Lambda_{j}^{2}+\Lambda_{k}^{2}\right)+\sum\limits_{k+1}^{j-1}\bigl((v_{z+1}-v_{z})\Lambda_{z}\bigr)^{2}\\
&+\sum\limits_{k+1}^{j-1}
\bigl((v_{l}-v_{j})(v_{z+1}-v_{z})\left(\Lambda_{j}^{2}+\Lambda_{z}^{2}\right)+\sum\limits_{k+1}^{j-1}
\bigl((v_{k+1}-v_{r})(v_{z+1}-v_{z})\left(\Lambda_{k}^{2}+\Lambda_{z}^{2}\right)\\
&+\sum\limits_{z,z'=k+1}^{j-1}
(v_{z+1}-v_{z})(v_{z'+1}-v_{z'})(\Lambda_{z}^2+\Lambda_{z'}^2)\\&\mbox{ by using the Young's inequality $a^2+b^2\ge2ab$ }\\
&\ge\bigl((v_{l}-v_{j})\Lambda_{j}\bigr)^{2}
+\bigl((v_{k+1}-v_{r})\Lambda_{k}\bigr)^{2}+2(v_{l}-v_{j})(v_{k+1}-v_{r})\Lambda_{j}\Lambda_{k}+\sum\limits_{k+1}^{j-1}
\bigl((v_{z+1}-v_{z})\Lambda_{z}\bigr)^{2}\\
&+\sum\limits_{k+1}^{j-1}2(v_{l}-v_{j})(v_{z+1}-v_{z})\Lambda_{j}\Lambda_{z}+\sum\limits_{k+1}^{j-1}2(v_{k+1}-v_{r})(v_{z+1}-v_{z})\Lambda_{k}\Lambda_{z}\\
&+\sum\limits_{z,z'=k+1}^{j-1}
2(v_{z+1}-v_{z})(v_{z'+1}-v_{z'})\Lambda_{z}\Lambda_{z'}\\
&=\left((v_{l}-v_{j})\Lambda_{j}
+(v_{k+1}-v_{r})\Lambda_{k}+\sum\limits_{k+1}^{j-1}(v_{z+1}-v_{z})\Lambda_{z}\right)^{2}.
\end{align*}
Hence, this gives the $|u_{l}-u_{r}|\ge(v_{l}-v_{j})\Lambda_{j}+(v_{k+1}-v_{r})\Lambda_{k}+\sum\limits_{k+1}^{j-1}(v_{z+1}-v_{z})\Lambda_{z}$. To prove the second estimate, we can use the fact that $\Lambda_{j+1}\le\Lambda_{s+1}\le\Lambda_{z}\le\Lambda_{k+1}$, and the following expression from the previous calculation,
\begin{align*}
   (u_{l}-u_{r})^2
&= \bigl((v_{l}-v_{j})\Lambda_{j}\bigr)^{2}
+\bigl((v_{k+1}-v_{r})\Lambda_{k}\bigr)^{2}+(v_{l}-v_{j})(v_{k+1}-v_{r})\left(\Lambda_{j}^{2}+\Lambda_{k}^{2}\right)\\
&+\sum\limits_{k+1}^{j-1}
\bigl((v_{z+1}-v_{z})\Lambda_{z}\bigr)^{2}+\sum\limits_{z,z'=k+1}^{j-1}
(v_{z+1}-v_{z})(v_{z'+1}-v_{z'})(\Lambda_{z}^2+\Lambda_{z'}^2)\\ 
&\ge(v_{l}-v_{j})^{2}\Lambda_{j+1}^2
+(v_{k+1}-v_{r})^{2}\Lambda_{j+1}^2+2(v_{l}-v_{j})(v_{k+1}-v_{r})\Lambda_{j+1}^2+\sum\limits_{k+1}^{j-1}
(v_{z+1}-v_{z})^{2}\Lambda_{j+1}^2\\
&+\sum\limits_{z,z'=k+1}^{j-1}
2(v_{z+1}-v_{z})(v_{z'+1}-v_{z'})\Lambda_{j+1}^2\\
&=(v_{l}-v_{r})^{2}\Lambda_{j+1}^2.
\end{align*}
similarly, we can also prove the other side of the inequality, thus
\begin{align*}
 (v_{l}-v_{r})^{2}\Lambda_{j+1}^2\le (u_{l}-u_{r})^2\le(v_{l}-v_{r})^{2}\Lambda_{K+1}^2,\\
 (v_{l}-v_{r})\Lambda_{j+1}\le|u_{l}-u_{r}|\le(v_{l}-v_{r})\Lambda_{k+1}.
\end{align*}
We omit the $S_{2}$ case, since it can be done in the similar fashion.
\end{proof}
Next lemma provides a quantitative estimates between the rarefaction and shock curves of the system \eqref{appr-p}.
\begin{lemma}\label{shock-rarefaction}
Let $S_i$ and $R_{i}$ be shock and rarefaction curves, respectively. Assume that $S_{i}$ and $R_{i}$ starts from left state $(v_{l}, u_{l})$. Assume that $(v_{r}, u_{r})$ and $(v_{r}, \tilde{u}_{r})$ are right states for $S_{i}$, and $R_{i}$ respectively, such that $v_{l}\in[v_{j},v_{j+1}]$ and $v_{r}\in[v_{k},v_{k+1}]$. Then
\begin{align*}
    |(u_{l}-u_{r})-(u_{l}-\tilde{u}_{r})|\le \frac{(v_{l}-v_{r})^{2}\left(\Lambda_{j}-\Lambda_{k}\right)^{2}}{(v_{l}-v_{j})\Lambda_{j}
+(v_{k+1}-v_{r})\Lambda_{k}+\sum\limits_{k+1}^{j-1}(v_{z+1}-v_{z})\Lambda_{z}}
\end{align*}
\end{lemma}
\begin{proof}
We prove that lemma for $1$-family curves, and the $2$-family case can be proved in the similar way. Let us define $A:=u_{l}-u_{r}$ and $B:=u_{l}-\tilde{u}_{r}$. From the definition of $S_{1}$ and $R_{1}$ and Lemma \ref{shock-bound}, we know that $A>B>0$ and 
\begin{align*}
    A^{2}&= \bigl((v_{l}-v_{j})\Lambda_{j}\bigr)^{2}
+\bigl((v_{k+1}-v_{r})\Lambda_{k}\bigr)^{2}+(v_{l}-v_{j})(v_{k+1}-v_{r})\left(\Lambda_{j}^{2}+\Lambda_{k}^{2}\right)+\sum\limits_{k+1}^{j-1}
\bigl((v_{z+1}-v_{z})\Lambda_{z}\bigr)^{2}\\
&+\sum\limits_{k+1}^{j-1}
\bigl((v_{l}-v_{j})(v_{z+1}-v_{z})\left(\Lambda_{j}^{2}+\Lambda_{z}^{2}\right)+\sum\limits_{k+1}^{j-1}
\bigl((v_{k+1}-v_{r})(v_{z+1}-v_{z})\left(\Lambda_{k}^{2}+\Lambda_{z}^{2}\right)\\
&+\sum\limits_{z,z'=k+1}^{j-1}
(v_{z+1}-v_{z})(v_{z'+1}-v_{z'})(\Lambda_{z}^2+\Lambda_{z'}^2)
\end{align*}
\begin{align*}
    B^{2}&= \bigl((v_{l}-v_{j})\Lambda_{j}\bigr)^{2}
+\bigl((v_{k+1}-v_{r})\Lambda_{k}\bigr)^{2}+2(v_{l}-v_{j})(v_{k+1}-v_{r})\Lambda_{j}\Lambda_{k}+\sum\limits_{k+1}^{j-1}
\bigl((v_{z+1}-v_{z})\Lambda_{z}\bigr)^{2}\\
&+2\sum\limits_{k+1}^{j-1}
\bigl((v_{l}-v_{j})(v_{z+1}-v_{z})\Lambda_{j}\Lambda_{z}+2\sum\limits_{k+1}^{j-1}
\bigl((v_{k+1}-v_{r})(v_{z+1}-v_{z})\Lambda_{k}\Lambda_{z}\\
&+2\sum\limits_{z,z'=k+1}^{j-1}
(v_{z+1}-v_{z})(v_{z'+1}-v_{z'})\Lambda_{z}\Lambda_{z'}
\end{align*}
now let us consider 
\begin{align*}
    A^2-B^2&=(v_{l}-v_{j})(v_{k+1}-v_{r})\left(\Lambda_{j}-\Lambda_{k}\right)^{2}+\sum\limits_{k+1}^{j-1}
\bigl((v_{l}-v_{j})(v_{z+1}-v_{z})\left(\Lambda_{j}-\Lambda_{z}\right)^{2}\\
&+\sum\limits_{k+1}^{j-1}\bigl((v_{k+1}-v_{r})(v_{z+1}-v_{z})\left(\Lambda_{k}-\Lambda_{z}\right)^{2}
+\sum\limits_{s,t=k+1}^{j-1}
(v_{z+1}-v_{z})(v_{z'+1}-v_{z'})\left(\Lambda_{z}-\Lambda_{z'}\right)^{2}\\
&\le\left( (v_{l}-v_{j})(v_{k+1}-v_{r})+\sum\limits_{k+1}^{j-1}
(v_{l}-v_{j})(v_{z+1}-v_{z})+\sum\limits_{k+1}^{j-1}(v_{k+1}-v_{r})(v_{z+1}-v_{z})\right)\left(\Lambda_{j}-\Lambda_{k}\right)^{2}\\&+\sum\limits_{s,t=k+1}^{j-1}
(v_{z+1}-v_{z})(v_{z'+1}-v_{z'})\left(\Lambda_{j}-\Lambda_{k}\right)^{2}\\
&\le(v_{l}-v_{r})^{2}\left(\Lambda_{j}-\Lambda_{k}\right)^{2}.
\end{align*}
Since $A+B>2B$, therefore,
\begin{align*}
    A-B\le \frac{(v_{l}-v_{r})^{2}\left(\Lambda_{j}-\Lambda_{k}\right)^{2}}{(v_{l}-v_{j})\Lambda_{j}
+(v_{k+1}-v_{r})\Lambda_{k}+\sum\limits_{k+1}^{j-1}(v_{z+1}-v_{z})\Lambda_{z}}.
\end{align*}
Hence we are done.
\end{proof}
Next lemma gives us the quantitative estimate between the shock curves starting from the same point and ends on different points.
\begin{lemma}\label{shock-shock}
Let $S_{i}, \tilde{S}_{i}$ and $\hat{S}_{i}$ be three shock curves. Assume that $S_{i}, \tilde{S}_{i}$ starts from left state $(v_{l}, u_{l})$ and $\hat{S}_{i}$ starts from left state $(\tilde{v}_{r}, \tilde{u}_{r})$  and $S_{i}, \tilde{S}_{i}$ and $\hat{S}_{i}$ ends at right states $(v_{r}, u_{r})$, $(\tilde{v}_{r}, \tilde{u}_{r})$, and $(v_{r}, \hat{u}_{r})$, respectively, with $v_{l}\in[v_{j},v_{j+1}]$, $\tilde{v}_{r}\in[v_{i},v_{i+1}]$ and $v_{r}\in[v_{k},v_{k+1}]$ such that $v_{l}>\tilde{v}_{r}>v_{r}$. Then,
\begin{align*}
     |u_{r}-\hat{u}_{r}|\le\frac{(v_{l}-\tilde{v}_{r})(\tilde{v}_{r}-v_{r})\left(\Lambda_{k}-\Lambda_{j}\right)^{2}}{2\left((v_{l}-v_{j})\Lambda_{j}
+(v_{k+1}-v_{r})\Lambda_{k}+\sum\limits_{k+1}^{j-1}(v_{z+1}-v_{z})\Lambda_{z}\right)}
\end{align*}
\end{lemma}
\begin{proof}
We prove that lemma for $1$-family curves, and the $2$-family case can be proved in the similar way. Let us define $A:=u_{l}-u_{r}$, $B:=u_{l}-\tilde{u}_{r}$ and $C:=\tilde{u}_{r}-\hat{u}_{r}$. From the definition of $S_{1}$ and Lemma \ref{shock-bound}, we know that $A>B+C>0$ and   $A^{2}=(v_{l}-v_{r})(p_{n}(v_{r})-p_{n}(v_{l})),
    B^{2}=(v_{l}-\tilde{v}_{r})(p_{n}(\tilde{v}_{r})-p_{n}(v_{l}))$, and $C^{2}=(\tilde{v}_{r}-v_{r})(p_{n}(v_{r})-p_{n}(\tilde{v}_{r}))$
\begin{align*}
  A^{2}-B^{2}-C^{2}=(v_{l}-\tilde{v}_{r})(p_{n}(v_{r})-p_{n}(\tilde{v}_{r}))+(\tilde{v}_{r}-v_{r})(p_{n}(\tilde{v}_{r})-p_{n}(v_{l}))
\end{align*}
 we can rewrite this as follows,
 \begin{align*}
    &A^{2}-B^{2}-C^{2}\\
    &=(v_{l}-v_{j})(\tilde{v}_{r}-v_{i})\left(\Lambda_{i}^{2}+\Lambda_{j}^{2}\right)+(v_{l}-v_{j})(v_{k+1}-v_{r})\left(\Lambda_{k}^{2}+\Lambda_{j}^{2}\right)+\sum\limits_{k+1}^{i-1}(v_{l}-v_{j})(v_{z+1}-v_{z})\left(\Lambda_{z}^{2}+\Lambda_{j}^{2}\right)\\
&+2(v_{i+1}-\tilde{v}_{r})(\tilde{v}_{r}-v_{i})\Lambda_{i}^{2}+(v_{i+1}-\tilde{v}_{r})(v_{k+1}-v_{r})\left(\Lambda_{i}^{2}+\Lambda_{k}^{2}\right)+\sum\limits_{k+1}^{i-1}(v_{i+1}-\tilde{v}_{r})(v_{z+1}-v_{z})\left(\Lambda_{i}^{2}+\Lambda_{z}^{2}\right)\\
&+\sum\limits_{i+1}^{j-1}(\tilde{v}_{r}-v_{i})(v_{z+1}-v_{z})\left(\Lambda_{i}^{2}+\Lambda_{z}^{2}\right)+\sum\limits_{i+1}^{j-1}(v_{k+1}-v_{r})(v_{z+1}-v_{z})\left(\Lambda_{z}^{2}+\Lambda_{k}^{2}\right)\\
&+\sum\limits_{i+1}^{j-1}\sum\limits_{k+1}^{i-1}(v_{z'+1}-v_{z'})(v_{z+1}-v_{z})\left(\Lambda_{z'}^{2}+\Lambda_{z}^{2}\right)
 \end{align*}
by using the definition of $B$, $C$, and Lemma \ref{shock-bound}, we know that
\begin{align*}
    2BC&\ge2\left((v_{l}-v_{j})\Lambda_{j}
+(v_{i+1}-\tilde{v}_{r})\Lambda_{i}+\sum\limits_{i+1}^{j-1}(v_{z+1}-v_{z})\Lambda_{z}\right)\\
&\hspace{60mm}\cdot\left((\tilde{v}_{r}-v_{i})\Lambda_{i}
+(v_{k+1}-v_{r})\Lambda_{k}+\sum\limits_{k+1}^{i-1}(v_{z+1}-v_{z})\Lambda_{z}\right)\\
&=2(v_{l}-v_{j})(\tilde{v}_{r}-v_{i})\Lambda_{i}\Lambda_{j}+2(v_{l}-v_{j})(v_{k+1}-v_{r})\Lambda_{k}\Lambda_{j}+2\sum\limits_{k+1}^{i-1}(v_{l}-v_{j})(v_{z+1}-v_{z})\Lambda_{j}\Lambda_{z}\\
&+2(v_{i+1}-\tilde{v}_{r})(\tilde{v}_{r}-v_{i})\Lambda_{i}^{2}+2(v_{i+1}-\tilde{v}_{r})(v_{k+1}-v_{r})\Lambda_{k}\Lambda_{i}+2\sum\limits_{k+1}^{i-1}(v_{i+1}-\tilde{v}_{r})(v_{z+1}-v_{z})\Lambda_{i}\Lambda_{z}\\
&+2\sum\limits_{i+1}^{j-1}(\tilde{v}_{r}-v_{i})(v_{z+1}-v_{z})\Lambda_{i}\Lambda_{z}+2\sum\limits_{i+1}^{j-1}(v_{k+1}-v_{r})(v_{z+1}-v_{z})\Lambda_{k}\Lambda_{z}\\
&+2\sum\limits_{i+1}^{j-1}\sum\limits_{k+1}^{i-1}(v_{z'+1}-v_{z'})(v_{z+1}-v_{z})\Lambda_{j}\Lambda_{z}
\end{align*}
hence,
\begin{align*}
    &A^{2}-(B+C)^2\\
    &\le(v_{l}-v_{j})(\tilde{v}_{r}-v_{i})\left(\Lambda_{i}-\Lambda_{j}\right)^{2}+(v_{l}-v_{j})(v_{k+1}-v_{r})\left(\Lambda_{k}-\Lambda_{j}\right)^{2}+\sum\limits_{k+1}^{i-1}(v_{l}-v_{j})(v_{z+1}-v_{z})\left(\Lambda_{z}-\Lambda_{j}\right)^{2}\\
&\scalebox{0.93}{$+(v_{i+1}-\tilde{v}_{r})(v_{k+1}-v_{r})\left(\Lambda_{i}-\Lambda_{k}\right)^{2}+\sum\limits_{k+1}^{i-1}(v_{i+1}-\tilde{v}_{r})(v_{z+1}-v_{z})\left(\Lambda_{i}-\Lambda_{z}\right)^{2}
+\sum\limits_{i+1}^{j-1}(\tilde{v}_{r}-v_{i})(v_{z+1}-v_{z})\left(\Lambda_{i}-\Lambda_{z}\right)^{2}$}\\
&+\sum\limits_{i+1}^{j-1}(v_{k+1}-v_{r})(v_{z+1}-v_{z})\left(\Lambda_{z}-\Lambda_{k}\right)^{2}
+\sum\limits_{i+1}^{j-1}\sum\limits_{k+1}^{i-1}(v_{z'+1}-v_{z'})(v_{z+1}-v_{z})\left(\Lambda_{z'}-\Lambda_{z}\right)^{2}\\
&\le\Bigl((v_{l}-v_{j})(\tilde{v}_{r}-v_{i})+(v_{l}-v_{j})(v_{k+1}-v_{r})+\sum\limits_{k+1}^{i-1}(v_{l}-v_{j})(v_{z+1}-v_{z})+(v_{i+1}-\tilde{v}_{r})(v_{k+1}-v_{r})\\
&+\sum\limits_{k+1}^{i-1}(v_{i+1}-\tilde{v}_{r})(v_{z+1}-v_{z})
+\sum\limits_{i+1}^{j-1}(\tilde{v}_{r}-v_{i})(v_{z+1}-v_{z})
+\sum\limits_{i+1}^{j-1}(v_{k+1}-v_{r})(v_{z+1}-v_{z})\\
&+\sum\limits_{i+1}^{j-1}\sum\limits_{k+1}^{i-1}(v_{z'+1}-v_{z'})(v_{z+1}-v_{z})\Bigr)\left(\Lambda_{k}-\Lambda_{j}\right)^{2}\\
&\le(v_{l}-\tilde{v}_{r})(\tilde{v}_{r}-v_{r})\left(\Lambda_{k}-\Lambda_{j}\right)^{2}.
\end{align*}
We know that $A+B+C>2A\ge2\left((v_{l}-v_{j})\Lambda_{j}
+(v_{k+1}-v_{r})\Lambda_{k}+\sum\limits_{k+1}^{j-1}(v_{z+1}-v_{z})\Lambda_{z}\right)$. This gives us,
\begin{align*}
    A-(B+C)\le\frac{(v_{l}-\tilde{v}_{r})(\tilde{v}_{r}-v_{r})\left(\Lambda_{k}-\Lambda_{j}\right)^{2}}{2\left((v_{l}-v_{j})\Lambda_{j}
+(v_{k+1}-v_{r})\Lambda_{k}+\sum\limits_{k+1}^{j-1}(v_{z+1}-v_{z})\Lambda_{z}\right)}
\end{align*}
\end{proof}
Now we define the strength for any front. We consider a front connecting states $U_{l}=(v_{l},u_{l})$ to $U_{r}=(v_{r},u_{r})$, then we define the wave strength,
\begin{equation}
  |\mu|=|u_{l}-u_{r}|.  
\end{equation}
Moreover, from \eqref{S1}-\eqref{R2}, we can see that $|u_{l}-u_{r}|=\mathcal{O}\left(|v_{l}-v_{r}|\right)$, provided $U_{l}, U_{r}$ are in some fixed compact set $K\subset\R^{2}$.
\begin{lemma}\label{lemmamtv}
Let $U(x,t)=(v(x,t), u(x,t))$ be a function of bounded variation. Define a functional $M(t)=\sum\limits_{k}|\mu_{k}|$, sum of all wave strength at time $t$, then
\begin{eqnarray}\label{mtv}
C_{1}M(t)\le TV(U(\cdot, t))\le C_{2}M(t)
\end{eqnarray}
where $C_{i}$, i=1, 2, are the positive constant depending only on the compact set $K$. 
\end{lemma}
\begin{proof}
Total variation of $U$ is given by,
\begin{align}
\TV\bigl(U(\cdot,t)\bigr)
&= \sum_{k} \abs{u_k - u_{k-1}}
  + \sum_{k} \abs{v_k - v_{k-1}} .
\end{align}
Therefore to prove the result it is enough to show that
\begin{align}\label{vvu}
\tilde{C} \abs{v_k - v_{k-1}}
\le \abs{u_k - u_{k-1}}
\le C \abs{v_k - v_{k-1}},
\end{align}
where $\tilde{C}, C$ depends only on compact set $K$.
We will show it for 1-family shock curve and other cases can be shown similar fashion. Consider 1-family shock curve connecting states $U_{l}=(v_{l},u_{l})$ to $U_{r}=(v_{r},u_{r})$, $v_{l}>v_{r}$ then
\begin{eqnarray*}
|u_{l}-u_{r}|=\sqrt{(v_{l}-v_{r})(p_{n}(v_{r})-p_{n}(v_{l})},
\end{eqnarray*}
from Lemma \ref{shock-bound} we can use the expanded shock expression and by using the fact $\Lambda_{z}$ are bound above and below on a compact set, therefore item \ref{Lambda_bound} gives us the \eqref{vvu}. Thus we are done.
\end{proof}

Now, in next subsections  we will present the interaction estimates required for obtain the uniform BV bound for the approximated weak entropy solutions.
\subsection{Different family interaction}
We consider three states \((v_l,u_l)\), \((v_m,u_m)\), and \((v_r,u_r)\). Assume that
\((v_l,u_l)\) and \((v_m,u_m)\) are connected by a \(2\)-family front, while
\((v_m,u_m)\) and \((v_r,u_r)\) are connected by a \(1\)-family front. After the
interaction, the outgoing \(1\)-family front connects \((v_l,u_l)\) to
\((v_m^*,u_m^*)\), and the outgoing \(2\)-family curve connects
\((v_m^*,u_m^*)\) to \((v_r,u_r)\). We denote the strengths of the incoming waves by
\[
|\mu_2| = |u_r-u_m|, \qquad |\mu_1| = |u_l-u_m|,
\]
and the strengths of the outgoing waves by
\[
|\widetilde{\mu}_1| = |u_l-u_m^*|, \qquad
|\widetilde{\mu}_2| = |u_m^*-u_r|.
\] 
\begin{itemize}
     \item[1.] {\bf ${\bf R_{2}+R_{1}}$ interaction:} For any generic front interaction $R_{2}$ and $R_{1}$.
     \[ |\mu_{1}|-|\tilde{\mu}_{1}|=0, \qquad \text{ and }\qquad
         |\mu_{2}|-|\tilde{\mu}_{2}|=0\]

\begin{center}
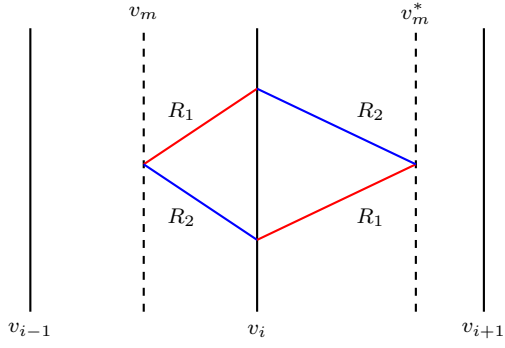

 
\begin{tikzpicture}[scale=1]

% Vertical solid lines
\draw[thick] (-3,1.25) -- (-3,5);   % v_{i-1}
\draw[thick] (0,1.25) -- (0,5);     % v_i
\draw[thick] (3,1.25) -- (3,5);     % v_{i+1}

% Dashed vertical lines
\draw[dashed, thick] (-1.5,1.25) -- (-1.5,5);   % v_m
\draw[dashed, thick] (2.1,1.25) -- (2.1,5);     % v_m^*

% Labels at top (dashed lines)
\node at (-1.5,5.2) {\scriptsize $v_m$};
\node at (2.1,5.2) {\scriptsize $v_m^\ast$};

% Bottom labels
\node at (-3,1) {\scriptsize $v_{i-1}$};
\node at (0,1) {\scriptsize $v_i$};
\node at (3,1) {\scriptsize $v_{i+1}$};
\node at (-1,2.5) {\scriptsize $R_{2}$};
\node at (-1,3.9) {\scriptsize $R_{1}$};
\node at (1.5,2.5) {\scriptsize $R_{1}$};
\node at (1.5,3.9) {\scriptsize $R_{2}$};
% Diamond interaction region
\coordinate (Top)    at (0,4.2);
\coordinate (Left)   at (-1.5,3.2);
\coordinate (Bottom) at (0,2.2);
\coordinate (Right)  at (2.1,3.2);

\draw[thick,red] (Top) -- (Left);
\draw[thick,blue] (Left)-- (Bottom);
\draw[thick, red](Bottom)-- (Right);
\draw[thick, blue](Right) -- (Top);
\end{tikzpicture}
\captionof{figure}{Illustration of interaction of rarefaction curves in $vu$-plane: $R_{1}, R_{2}$ are in red and blue, respectively.} \label{fig-1}
\end{center}
\item[2.]{\bf ${\bf S_{2}+S_{1}}$ interaction:} Assume $S_{2}$ connects $(v_{l}, u_{l})$ to $(v_{m}, u_{m})$,   $S_{1}$ connects $(v_{m}, u_{m})$ to $(v_{r}, u_{r})$ are interacting, where $v_{l}\in[v_{j-1}, v_{j}], v_{m}\in[v_{i-1},v_{i}]$. The outgoing $S_{1}$ connects $(v_{l}, u_{l})$ to $(v_{m}^{*}, u_{m}^{*})$, outgoing   $S_{2}$ connects $(v_{m}^{*}, u_{m}^{*})$ to $(v_{r}, u_{r})$, where $v_{r}\in[v_{k-1},v_{k}], v_{m}^{*}\in[v_{i-\alpha+1},v_{i-\alpha}]$ (see Fig \ref{Fig-shock-shock}). 
\begin{center}
    \begin{tikzpicture}[scale=.9]

% =====================================================
% Five solid vertical lines
% =====================================================
\foreach \x in {-11, -10,-7.7,-6,-3.5,-1,1,4}
  \draw[thick] (\x,0) -- (\x,5.5);
\foreach \x in {-10.5,-6.8,-2,1.8}
  \draw[dashed] (\x,0) -- (\x,5.5);

% =====================================================
% Labels
% =====================================================
\node at (-11.15,-0.25) {\scriptsize $v_{i-\alpha}$};
\node at (4,-0.25) {\scriptsize $v_{i}$};
\node at (1,-0.25) {\scriptsize $v_{i-1}$};
\node at (-1,-0.25) {\scriptsize $v_{j}$};
\node at (-3.5,-0.25) {\scriptsize $v_{j-1}$};
\node at (-6,-0.25) {\scriptsize $v_{k}$};
\node at (-7.7,-0.25) {\scriptsize $v_{k}$};
\node at (-2,-0.25) {\scriptsize $v_l$};
\node at (-10.45,-0.25) {\scriptsize $v_m^{*}$};
\node at (1.8,-0.25)  {\scriptsize $v_m$};
\node at (-6.8,-0.25)  {\scriptsize $v_r$};

% =====================================================
% Top S1 curve
% =====================================================
\coordinate (Tstart) at (-2,4.7);
\coordinate (Tend)   at (-10.5,2.3);

\draw[thick,blue]
  (Tstart)
  .. controls (-3.3,4.6) and (-6,4.4) ..
  (Tend)
  node[midway,above] {\scriptsize $S_1$};

% =====================================================
% Bottom S1 curve
% =====================================================
\coordinate (Bstart) at (1.8,3.8);
\coordinate (Bend)   at (-6.8,1.65);

\draw[thick,blue]
  (Bstart)
  .. controls (1.2,3.7) and (-4,3.5) ..
  (Bend)
  node[midway,below] {\scriptsize $S_1$};

% =====================================================
% Convex (almost straight) red connections between ends
% =====================================================

% Right ends: very slightly convex to the right
\draw[thick,red]
  (Tstart)
  .. controls (-1,4.2) and (1,3.9) ..
  (Bstart);

% Left ends: very slightly convex to the left
\draw[thick,red]
  (Tend)
  .. controls (-10.2,2.2) and (-8.9,1.75) ..
  (Bend);

% =====================================================
% Ellipsis
% =====================================================
\node at (0,3) {\Huge $\cdots$};
\node at (-4.6,4) {\Huge $\cdots$};
\node at (-8.8,4) {\Huge $\cdots$};

\end{tikzpicture}

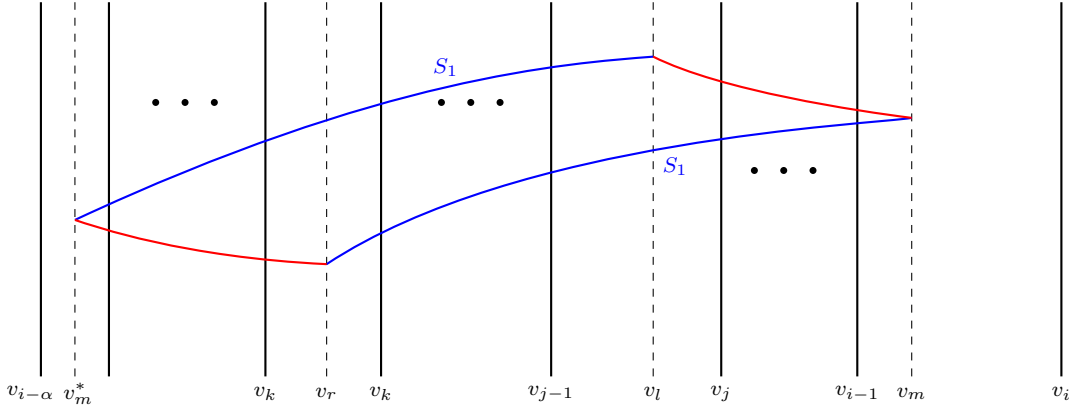
\captionof{figure}{Illustration of interaction of shock curves in $vu$-plane: $S_{2}, S_{1}$ are in red and blue, respectively.} \label{Fig-shock-shock}
\end{center}

Consider a point $\bar{u}_{m}>u_{m}$ on line $v_{m}$ such that
\begin{align}\label{delta}
  \bar{u}_{m}-u_{l}=u_{l}-u_{m}=\sqrt{(v_{m}-v_{l})(p_{n}(v_{l})-p_{n}(v_{m}))}:=\delta 
\end{align}
Hence, we can connect $(v_{m}, \bar{u}_{m})$ to $(v_{l}, u_{l})$ by a 1-family shock curve, say, $S_{1}^{(1)}$. Next, we consider another  1-family shock curve, say, $S_{1}^{(2)}$ from $(v_{l}, u_{l})$ to $(v_{r}, \bar{u}_{r})$. By using the Lemma \ref{shock-shock}, we know
\begin{align}\label{uvsu}
    (u_{m}-u_{r})-(\bar{u}_{m}-\bar{u}_{r})\le C(v_{m}-v_{l})(v_{l}-v_{r})(v_{m}-v_{r}),
\end{align}
using Lemma \ref{shock-bound} and \eqref{delta}, we can write \eqref{uvsu},
\begin{align}\label{uvsu-delta}
    (u_{m}-u_{r})-(\bar{u}_{m}-\bar{u}_{r})\le C(u_{l}-u_{m})(u_{l}-\bar{u}_{r})(u_{m}-u_{r}):=C\delta_{1},
\end{align}
where $\delta_{1}=(u_{l}-u_{m})(u_{l}-\bar{u}_{r})(u_{m}-u_{r})$.
We draw a shock curve $S_{1}^{(3)}$ from $(v_{r},\bar{u}_{r})$ to some point $(v',u')$ such that $\bar{u}_{r}-u'=\bar{u}_{m}-u_{l}=\delta$. Next, we draw another 1-family shock curves, say, $S_{1}^{(4)}$ from point $(v_{l},u_{l})$ to $(v', u'')$, from Lemma \ref{shock-shock}, we get
\begin{align}
    (u_{l}-u'')-(u_{l}-u')=u'-u''\le C(u_{l}-u_{m})(u_{l}-\bar{u}_{r})(u_{m}-u_{r})=C\delta_{1},
\end{align}
by our choice of $S_{1}^{(i)}$ for $1\le i\le4$, we can see that $\bar{u}_{m}-\bar{u}_{r}=u_{l}-u'$.
Now we estimate $|(u_{l}-u'')-(u_{m}-u_{r})|$,
\begin{align}\label{ledelta}
  (u_{l}-u'')-(u_{m}-u_{r})&=\underbrace{(u_{l}-u'')-(\bar{u}_{m}-\bar{u}_{r})}_{\le C\delta_{1}}+\underbrace{(\bar{u}_{m}-\bar{u}_{r})-(u_{m}-u_{r})}_{\le0}\le C\delta_{1},\\\label{gedelta}
  (u_{l}-u'')-(u_{m}-u_{r})&=\underbrace{(u_{l}-u'')-(\bar{u}_{m}-\bar{u}_{r})}_{\ge0}+\underbrace{(\bar{u}_{m}-\bar{u}_{r})-(u_{m}-u_{r})}_{\ge-C\delta_{1}}\ge -C\delta_{1} ,
\end{align}
from \eqref{ledelta}-\eqref{gedelta}, we can conclude,
\begin{align}\label{u''-ur}
    \delta-C\delta_{1}\le u''-u_{r}\le\delta+C\delta_{1}.
\end{align}
Now, assume that if $u_{m}^{*}>u''$, then $v_{m}^{*}>v'$, which implies that $S_{2}$ start from $(v_{m}^{*},u_{m}^{*})$ to $(v_{r},u_{r})$ satisfies $u_{m}^{*}-u_{r}<\delta$, hence together with \eqref{u''-ur} we get, $0\le u_{m}^{*}-u''\le C\delta_{1}$. Similarly, if $u_{m}^{*}<u''$, we will get, $-C\delta_{1}\le u_{m}^{*}-u''\le 0$. We can estimate $|(u_{l}-u_{m}^{*})-(u_{m}-u_{r})|$,
\begin{align}\label{fin-est1}
  (u_{l}-u_{m}^{*})-(u_{m}-u_{r})&=\underbrace{(u_{l}-u'')-(u_{m}-u_{r})}_{\le C\delta_{1}}+\underbrace{(u''-u_{m}^{*})}_{\le C\delta_{1}}\le 2C\delta_{1},\\\label{fin-est2}
  (u_{l}-u_{m}^{*})-(u_{m}-u_{r})&=\underbrace{(u_{l}-u'')-(u_{m}-u_{r})}_{\ge -C\delta_{1}}+\underbrace{(u''-u_{m}^{*})}_{\ge -C\delta_{1}}\ge -2C\delta_{1} ,
\end{align}
hence, from \eqref{fin-est1}-\eqref{fin-est2} we get the estimate, 
\[ \left| \tilde{\mu}_{1}-\mu_{1}\right|\le C|\mu_{1}||\mu_{2}|(|\mu_{1}|+|\mu_{2}|), \qquad\text{ and } \qquad
     \left| \tilde{\mu}_{2}-\mu_{2}\right|\le C|\mu_{1}||\mu_{2}|(|\mu_{1}|+|\mu_{2}|).\]

\item[3.]{\bf ${\bf R_{2}+S_{1}/S_{2}+R_{1}}$ interaction:}  Let us assume $R_{2}$ connects $(v_{l}, u_{l})$ to $(v_{m}, u_{m})$,   $S_{1}$ connects $(v_{m}, u_{m})$ to $(v_{r}, u_{r})$ are interacting, where $v_{l},v_{m}\in[v_{i-1},v_{i}]$. The outgoing $S_{1}$ connects $(v_{l}, u_{l})$ to $(v_{m}^{*}, u_{m}^{*})$, outgoing   $R_{2}$ connects $(v_{m}^{*}, u_{m}^{*})$ to $(v_{r}, u_{r})$, where $v_{r}, v_{m}^{*}\in[v_{i-\alpha},v_{i-\alpha+2}]$ (see Fig \ref{fig-2}). Consider a shock curve $\tilde{S}_{1}$ starts from $(v_{l}, u_{l})$ to $(v_{r}, \bar{u}_{r})$. If we choose $\tilde{u}_{l}>u_{m}$ such that $\tilde{u}_{l}-u_{m}=u_{m}-u_{l}=(v_{l}-v_{m})\sqrt{-p'(v_{i-1}^{n})}$, then from Lemma \ref{shock-bound}, we know that $u_{l}-\bar{u}_{r}\ge\tilde{u}_{l}-u_{r}$ and Lemma \ref{shock-rarefaction} gives us
\begin{align*}
    (u_{l}-\bar{u}_{r})-(\tilde{u}_{l}-u_{r})\le C(v_{l}-v_{m})(v_{m}-v_{r})(v_{l}-v_{r}),
\end{align*}
again by Lemma \ref{shock-bound}, we have $u_{m}^{*}-\bar{u}_{r}\ge(v_{l}-v_{m})\sqrt{-p'(v_{i-1}^{n})}=\tilde{u}_{l}-u_{m}=u_{m}-u_{l}$. Thus, we get,
\begin{align}\label{r2s1}
    (u_{l}-u_{m}^{*})-(u_{m}-u_{r})\le C(v_{l}-v_{m})(v_{m}-v_{r})(v_{l}-v_{r}).
\end{align}
Now \eqref{r2s1} and \eqref{vvu}, gives the estimate
% We denote $|\tilde{\mu}_{1}|=|u_{l}-u_{m}^{*}|$, and $|\tilde{\mu}_{2}|=|u_{m}^{*}-u_{r}|$ are the wave strength of outgoing $S_{1}$ and $R_{2}$, respectively. By using the Lemma \ref{shock-rarefaction}, we can estimate, 
\[ \left| \tilde{\mu}_{1}-\mu_{1}\right|\le C|\mu_{1}||\mu_{2}|(|\mu_{1}|+|\mu_{2}|), \qquad\text{ and } \qquad
     \left| \tilde{\mu}_{2}-\mu_{2}\right|\le C|\mu_{1}||\mu_{2}|(|\mu_{1}|+|\mu_{2}|).\]
\begin{center}
\begin{tikzpicture}[scale=1]

% =====================================================
% Five solid vertical lines
% =====================================================
\foreach \x in {-6.2,-4.8,-3,0,2.5}
  \draw[thick] (\x,1) -- (\x,5);
  \foreach \x in {-5.2,-4.5,0.5,2.2}
  \draw[dashed] (\x,1) -- (\x,5);

% =====================================================
% Labels
% =====================================================
\node at (-6.2,.75) {\scriptsize $v_{i-\alpha}$};
\node at (2.5,.75) {\scriptsize $v_{i}$};
\node at (0,.75) {\scriptsize $v_{i-1}$};
\node at (-5.2,.75) {\scriptsize $v_r$};
\node at (-4.5,.75) {\scriptsize $v_m^{*}$};
\node at (-3,.75) {\scriptsize $v_{i-\alpha+2}$};
\node at (0.6,.75)  {\scriptsize $v_m$};
\node at (2.2,.75)  {\scriptsize $v_l$};

% =====================================================
% Top S1 curve (v_m -> v_l)
% =====================================================
\coordinate (Tstart) at (0.5,4.6);
\coordinate (Tend)   at (-5.2,2.6); % left end unchanged

% Adjusted control points for monotonically decreasing concave curve
\draw[thick,blue]
  (Tstart)
  .. controls (0.3,4.5) and (-2.5,4.3) ..
  (Tend)(Tend)
   node[midway,above] {\scriptsize $S_1$};

% =====================================================
% Bottom S1 curve (v_r -> v_m^*)
% =====================================================
\coordinate (Bstart) at (2.2,4.3); % right end unchanged
\coordinate (Bend)   at (-4.5,1.65); % left end unchanged

% Adjusted control points for monotonically decreasing concave curve
\draw[thick,blue]
  (Bstart)
  .. controls (1.2,4.1) and (-2.2,3.5) ..
  (Bend)
  node[midway,below] {\scriptsize $S_1$};

% =====================================================
% Straight connections between ends
% =====================================================
% Right ends
\draw[thick,red] (Tstart) -- (Bstart);
% Left ends
\draw[thick,red] (Tend) --(-4.8,1.9)-- (Bend);

% =====================================================
% Ellipsis
% =====================================================
\node at (-1,3) { $\cdots$};

\end{tikzpicture}
\captionof{figure}{Illustration of interaction of rarefaction curves in $vu$-plane: $R_{2}, S_{1}$ are in red and blue, respectively.} \label{fig-2}
\end{center}
\end{itemize}
% \begin{center}
% \begin{tikzpicture}

% % Vertical grid lines (all parallel)
% \draw[thick] (-4,0) -- (-4,6);
% \draw[thick] (-2.5,0) -- (-2.5,6);
% \draw[thick] (-1,0) -- (-1,6);
% \draw[thick] (0.5,0) -- (0.5,6);
% \draw[thick] (3,0) -- (3,6);

% % Bottom labels
% \node at (-4,-0.25) {\scriptsize $v_{i-2}$};
% \node at (-2.5,-0.25) {\scriptsize $v_{i-1}$};
% \node at (-1,-0.25) {\scriptsize $v_i$};
% \node at (0.5,-0.25) {\scriptsize $v_{i+1}$};
% \node at (3,-0.25) {\scriptsize $v_{i+2}$};

% % Dashed vertical characteristic vp (parallel)
% \draw[dashed, thick] (-0.25,0) -- (-0.25,6);
% \node at (-0.25,6.25) {\scriptsize $v_r$};
% \node at (0.25,6.25) {\scriptsize $v_l$};
% \node at (-0.75,6.25) {\scriptsize $v_m$};

% % Dashed parallel boundaries inside the cell
% \draw[dashed] (-0.75,0) -- (-0.75,6);
% \draw[dashed] (0.25,0) -- (0.25,6);
% \draw[red] (.25,2)--(-.75,3);
% \draw[blue] (-.25,3.5)--(-.75,3);
% \node at (-.25,3.5) {\tiny $\bullet$};
% \node at (.25,2) {\tiny $\bullet$};
% \node at (-.75,3) {\tiny $\bullet$};
% \node at (-.25,2.5) {\tiny $\bullet$};
% \end{tikzpicture}    
% \end{center}

\subsection{Same family interaction}
Consider three states \((v_l,u_l)\), \((v_m,u_m)\), and \((v_r,u_r)\).
Assume that \((v_l,u_l)\) is connected to \((v_m,u_m)\), and
\((v_m,u_m)\) is connected to \((v_r,u_r)\), by two front of the same family,
say the \(i\)-family, with \(i\in\{1,2\}\). We denote by
$|\mu_i|\text{ and }|\mu_i'|$
the strengths of the two incoming \(i\)-family waves. After the interaction,
the outgoing waves consist of an \(i\)-family wave and a \((3-i)\)-family wave.
Their strengths are denoted by $|\widetilde{\mu}_i|\text{ and } |\widetilde{\mu}_{3-i}|,$
respectively. 
   \begin{itemize}
     \item[1.] {\bf ${\bf S_{1}+S_{1}/S_{2}+S_{2}}$ interaction:} In the $S_{i}+S_{i}$, for $i=1,2$ interaction the outgoing waves are $R_{3-i}$ and $S_{i}$. Interaction estimate follows directly from  Lemma \ref{shock-rarefaction} and Lemma \ref{shock-shock}, we can estimate, 
\begin{align*}
     \left| \tilde{\mu}_{i}-\mu_{i}-\mu_{i}'\right|+|\tilde{\mu}_{3-i}|&\le C|\mu_{i}||\mu_{i}'|(|\mu_{i}|+|\mu_{i}'|)
\end{align*}
   \begin{center}
    \begin{tikzpicture}[scale=1]

% =====================================================
% Five solid vertical lines
% =====================================================
\foreach \x in {-9.5,-8,-3.5,-1,1,4}
  \draw[thick] (\x,1) -- (\x,6);
\foreach \x in {-9,-2,1.8, 3.5}
  \draw[dashed] (\x,1) -- (\x,6);

% =====================================================
% Labels
% =====================================================
\node at (-9.5,.75) {\scriptsize $v_{k}$};
\node at (4,.75) {\scriptsize $v_{i}$};
\node at (1,.75) {\scriptsize $v_{i-1}$};
\node at (-1,.75) {\scriptsize $v_{j}$};
\node at (-3.5,.75) {\scriptsize $v_{j-1}$};
\node at (-8,.75) {\scriptsize $v_{k+1}$};
\node at (-2,.75) {\scriptsize $v_m$};
\node at (-9,.75) {\scriptsize $v_r$};
\node at (1.8,.75)  {\scriptsize $v_l$};
\node at (3.5,.75)  {\scriptsize $v_m^{*}$};
% =====================================================
% Top S1 curve
% =====================================================
\coordinate (Tstart) at (-2,4.7);
\coordinate (Tend)   at (-9.,1.3);

\draw[thick,blue]
  (Tstart)
  .. controls (-3.3,4.6) and (-6,4.4) ..
  (Tend)
  node[midway,above] {\scriptsize $S_1$};

% =====================================================
% Bottom S1 curve
% =====================================================
\coordinate (Bstart) at (1.8,5.8);
\coordinate (Bend)   at (-2,4.7);

\draw[thick,blue]
  (Bstart)
  .. controls (1.2,5.7) and (.7,5.8) ..
  (Bend)
  node[midway,below] {\scriptsize $S_1$};
  
  % =====================================================
% Bottom S1 curve
% =====================================================
\coordinate (Bstart2) at (3.5,5);
\coordinate (Bend2)   at (-9.2,1.2);

\draw[thick,blue]
  (Bstart2)
  .. controls (1.2,4.9) and (-4,4.7) ..
  (Tend)
  node[midway,below] {\scriptsize $S_1$};

% =====================================================
% Convex (almost straight) red connections between ends
% =====================================================

%Right ends: very slightly convex to the right
\draw[thick,red]
  (Bstart)--(Bstart2);

%Left ends: very slightly convex to the left
% \draw[thick,red]
%   (Tend)--(-10,1.5)-- (Bend2);

% =====================================================
% Ellipsis
% =====================================================
 \node at (0,3) {\Huge $\cdots$};
% \node at (-4.6,4) {\Huge $\cdots$};
 \node at (-7,4) {\Huge $\cdots$};

\end{tikzpicture}
\captionof{figure}{Illustration of interaction of shock curves in $vu$-plane: $S_{1}, S_{1}$ are in red and blue, respectively.} \label{fig-3}
\end{center} 
\item[2.] {\bf ${\bf S_{1}+R_{1}/S_{2}+R_{2}}$ interaction:}In the $S_{i}+R_{i}$ interaction the outgoing waves are $S_{3-i}$ and $S_{i}$. Interaction estimate follows directly from  Lemma \ref{shock-rarefaction} and Lemma \ref{shock-shock}, we can estimate, 
\begin{align*}
     \left| \tilde{\mu}_{i}-\mu_{i}-\mu_{i}'\right|+|\tilde{\mu}_{3-i}|&\le C|\mu_{i}||\mu_{i}'|(|\mu_{i}|+|\mu_{i}'|)
\end{align*}
\begin{center}
    \begin{tikzpicture}[scale=1]

% =====================================================
% Five solid vertical lines
% =====================================================
\foreach \x in {-9.5,-8,-3.5,-1}
  \draw[thick] (\x,2) -- (\x,6);
\foreach \x in {-9,-8.32,-2}
  \draw[dashed] (\x,2) -- (\x,6);

% =====================================================
% Labels
% =====================================================
\node at (-9.5,1.75) {\scriptsize $v_{k}$};
\node at (-1,1.75) {\scriptsize $v_{j}$};
\node at (-3.5,1.75) {\scriptsize $v_{j-1}$};
\node at (-8,1.75) {\scriptsize $v_{k+1}$};
\node at (-2,1.75) {\scriptsize $v_l$};
\node at (-9,1.75) {\scriptsize $v_m$};
\node at (-8.32,6.25)  {\scriptsize $v_m^{*}$};
% =====================================================
% Top S1 curve
% =====================================================
\coordinate (Tstart) at (-2,5.7);
\coordinate (Tend)   at (-9,2.3);

\draw[thick,blue]
  (Tstart)
  .. controls (-3.3,5.6) and (-6,5.4) ..
  (Tend)
  node[midway,above] {\scriptsize $S_1$};

% =====================================================
% Bottom S1 curve
% =====================================================
% \coordinate (Bstart) at (1.8,5.8);
% \coordinate (Bend)   at (-2,4.7);

% \draw[thick,blue]
%   (Bstart)
%   .. controls (1.2,5.7) and (.7,5.8) ..
%   (Bend)
%   node[midway,below] {\scriptsize $S_1$};
  
  % =====================================================
% Bottom S1 curve
% =====================================================
% \coordinate (Bstart2) at (3.5,5);
% \coordinate (Bend2)   at (-9.2,1.2);

% \draw[thick,blue]
%   (Bstart2)
%   .. controls (1.2,4.9) and (-4,4.7) ..
%   (Tend)
%   node[midway,below] {\scriptsize $S_1$};

% =====================================================
% Convex (almost straight) red connections between ends
% =====================================================

%Right ends: very slightly convex to the right
\draw[thick,red]
  (Tend)--(-8,2.8);
  \draw[dashed,red]
  (-8,2.8)--(-8.32,2.955);
%Left ends: very slightly convex to the left
% \draw[thick,red]
%   (Tend)--(-10,1.5)-- (Bend2);

% =====================================================
% Ellipsis
% =====================================================
% \node at (0,3) {\Huge $\cdots$};
% \node at (-4.6,4) {\Huge $\cdots$};
 \node at (-7,3) {\Huge $\cdots$};

\end{tikzpicture}

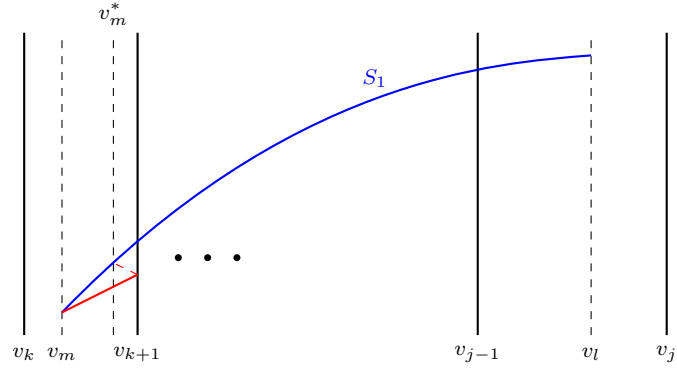
\captionof{figure}{Illustration of interaction of shock curves in $vu$-plane: $S_{1}, R_{1}$ are in red and blue, respectively.} \label{fig-3}
\end{center}
\end{itemize}
\subsection{Interaction with nonphysical front $\Gamma_{np}$}
Consider three states \((v_l,u_l)\), \((v_m,u_m)\), and \((v_r,u_r)\). Assume that $\beta$ and $\gamma$ are the two incoming shock and rarefaction fronts of $i$-th family, respectively.  Let
\[
        \beta:\ (v_{l}, u_{l})\to (v_m,u_m),\qquad
        \gamma:\ (v_m,u_m)\to (v_r,u_r),
\]
since $\gamma$ is a rarefaction front then $v_{m}, v_{r}\in[v_{i}, v_{i+1}]$ for some $i\in\mathbb{Z}$. After the interaction the outgoing fonts $i$-th family shock $\beta'$ and nonphysical shock $\Gamma_{np}$ are given 
\[
        \beta':\ (v_{l}, u_{l})\to (\hat{v}_r,\hat{u}_r),\qquad
        \Gamma_{np}:\ (\hat{v}_r,\hat{u}_r)\to (v_r,u_r),
\]
where \((\hat{v}_r, \hat{u}_{r})=(v_{r}, u_{l}-\sqrt{(p_{n}(v_{m})-p_{n}(v_{l}))(v_{l}-v_{r})})\). Set incoming front strength fronts $\beta$ and $\gamma$ are $|\mu_{\beta}|=|u_{m}-u_{l}|$ and $|\mu_{\gamma}|=|u_{r}-u_{m}|$, respectively.
 % Assume that $\beta$ and $\gamma$ are the two incoming shock and rarefaction fronts of $1$-family, respectively.  Let
% \[
%         \beta:\ (v_{l}, u_{l})\to (v_m,u_m),\qquad
%         \gamma:\ (v_m,u_m)\to (v_r,u_r),
% \]
% since $\gamma$ is a rarefaction front then $v_{m}, v_{r}\in[v_{i}, v_{i+1}]$ for some $i\in\mathbb{Z}$. Hence,
% \[\beta: u_{m}-u_{l}=\sqrt{(p_{n}(v_{m})-p_{n}(v_{l}))(v_{l}-v_{r})},\qquad
% \gamma:u_{r}-u_{m}=\Lambda(v_{r}-v_{m}),
% \]
% where $\Lambda=\sqrt{\frac{p(v_{i})-p(v_{i+1})}{v_{i+1}-v_{i}}}$.
% Now assume that the interaction is resolved by the modified Riemann solver, thus outgoing fonts $1$-family shock $\beta'$ and nonphysical shock $\Gamma_{np}$ are given 
% \[
%         \beta':\ (v_{l}, u_{l})\to (\hat{v}_r,\hat{u}_r),\qquad
%         \Gamma_{np}:\ (\hat{v}_r,\hat{u}_r)\to (v_r,u_r),
% \]
% where \((\hat{v}_r, \hat{u}_{r})=(v_{r}, u_{l}-\sqrt{(p_{n}(v_{m})-p_{n}(v_{l}))(v_{l}-v_{r})})\). Set incoming front strength fronts $\beta$ and $\gamma$ are $|\mu_{\beta}|=|u_{m}-u_{l}|$ and $|\mu_{\gamma}|=|u_{r}-u_{m}|$, respectively. 
Then directly from Lemma \ref{shock-rarefaction} and Lemma \ref{shock-shock}, we get 
\begin{align}\label{non-phy-est}
    |\mu_{np}|=|\hat{u}_{r}-u_{r}|\le C|\mu_{\beta}||\mu_{\gamma}|.
\end{align}
  In the next section, we will establish the uniform BV bound for approximate  weak solution to system \eqref{appr-p}. 
\section{Uniform BV bound}\label{uni-bv}
We consider the Glimm functional defined as follows
\begin{align}
    G(t)&=\sum\limits_{i}|\mu_{i}|+\theta\sum\limits_{(\mu_{k}, \mu_{j})\in A}|\mu_{k}||\mu_{j}|,
\end{align}
 where $\theta$ is a positive constant and $A$ is the approaching set. We show that $G(t)$ is non increasing in time. % and $f(|\mu_{k}|, |\mu_{j}|)$ is defined as $f(|\mu_{k}|, |\mu_{j}|)=0$ if $|\mu_{k}|=|\mu_{k}'|$ and $|\mu_{j}|=|\mu_{j}'|$,  otherwise $f(|\mu_{k}|, |\mu_{j}|)=|\mu_{k}||\mu_{j}|(|\mu_{k}|+|\mu_{j}|)$.
 Assume that at time $t$, two wave fronts interact with strengths $|\mu_{1}|$ and $|\mu_{2}|$ and the outgoing wave fronts strengths are $|\mu_{1}'|$ and $|\mu_{2}'|$. Therefore,
\begin{align*}
   G(t-)&=|\mu_{1}|+|\mu_{2}|+\sum\limits_{i}|\mu_{i}|+\theta |\mu_{1}||\mu_{2}|+\theta\sum\limits_{A}|\mu_{k}||\mu_{1}|+\theta\sum\limits_{A}|\mu_{\hat{k}}| |\mu_{2}|+\theta\sum\limits_{A}|\mu_{k}||\mu_{j}|\\
   G(t+)&=|\mu_{1}'|+|\mu_{2}'|+\sum\limits_{i}|\mu_{i}|+\theta\sum\limits_{A}|\mu_{k}||\mu_{1}'|+\theta\sum\limits_{A}|\mu_{\hat{k}}| |\mu_{2}'|+\theta\sum\limits_{A}|\mu_{k}||\mu_{j}|,
\end{align*}
Hence, we can write
\begin{align*}
    G(t+)-G(t-)&=(|\mu_{1}'|-|\mu_{1}|)+(|\mu_{2}'|-|\mu_{2}|)-\theta|\mu_{1}||\mu_{2}|+\theta\sum\limits_{A}|\mu_{k}|(|\mu_{1}'|-|\mu_{1}|)\\&+\theta\sum\limits_{A}|\mu_{\hat{k}}|( |\mu_{2}'|-|\mu_{2}|),
\end{align*}
from the analysis of interaction estimate in previous subsection, we get
\begin{align*}
    G(t+)-G(t-)&\le C|\mu_{1}||\mu_{2}|(|\mu_{1}|+|\mu_{2}|)-\theta|\mu_{1}||\mu_{2}|+C\theta|\mu_{1}||\mu_{2}|\left((|\mu_{1}|+|\mu_{2}|)\sum\limits_{A}|\mu_{k}|+|\mu_{\hat{k}}|\right),%\\&+C\theta|\mu_{1}||\mu_{2}|\sum\limits_{A}|\mu_{\hat{k}}|(|\mu_{1}|+|\mu_{2}|)
\end{align*}
now assume that $\left((|\mu_{1}|+|\mu_{2}|)\sum\limits_{A}|\mu_{k}|+|\mu_{\hat{k}}|\right):=\delta_{0}$ such that $C\delta_{0}=\frac{1}{2}$, hence this gives us
\begin{align*}
   G(t+)-G(t-)&\le C|\mu_{1}||\mu_{2}|(|\mu_{1}|+|\mu_{2}|)-\frac{\theta}{2}|\mu_{1}||\mu_{2}|\\
   &\le\left(C(|\mu_{1}|+|\mu_{2}|)-\frac{\theta}{2}\right)|\mu_{1}||\mu_{2}|
\end{align*}
by choosing the constant $\theta$ large enough, this gives us $ G(t+)\le G(t-)$. From \eqref{mtv} of Lemma \ref{lemmamtv}, we have
  $ TV(v(t+), u(t+))\le C_{2}M(t+)\le C_{2}G(t+)\le C_{2}G(0)$, note that $G(0)\le M(0)+\theta M(0)^{2}$. Now, assume that $TV(v(0),u(0))=\delta'$, then $TV(v(t+), u(t+))\le C_{2}\left(\frac{\delta'}{C_{1}}+\theta\frac{\delta'^2}{C_{1}^2}\right)$, by choosing the $\delta'$ such that $C_{2}\left(\frac{\delta'}{C_{1}}+\theta\frac{\delta'^2}{C_{1}^2}\right)\le\delta_{0}$. Hence, this gives us the uniform bound on the total variation of the weak solution to \eqref{appr-p}.
 \begin{proposition}
The construction of front tracking approximate weak solution obtained from our scheme must have finite number of fronts for all time.  
\end{proposition}
\begin{proof}
    For the proving the number of fronts are finite, we start with interactions of same family shock-rarefaction interactions. Let $\beta$ and $\gamma$ are the shock and rarefaction fronts of $i$-th family with strength $|\mu_{\beta}|$ and $|\mu_{\gamma}|$, respectively. Note that at every such interaction the Glimm functional $G(t)$ decrease by the $\kappa|\mu_{\beta}||\mu_{\gamma}|$, for some fix $0<\kappa<1$. Since these interaction can produce a new shock front of other family only when $\kappa|\mu_{\beta}||\mu_{\gamma}|\ge h_{n}$. Due to the fact that the Glimm functional is uniformly bounded by $G(0)$ and decrease at least by $h_{n}$ at such interactions, thus the total number of shock fronts remains finite. Consequently, we have nonphysical fronts only on these interactions, therefore the number of nonphysical fronts are also finite. The interactions of the restricted class handled by the counting functional $X(t)$. Since the number of fronts are shock fronts are finite, therefore, from analysis in subsection \ref{decay-count-poly}, we can have only finitely many such interactions. Therefore the total of fronts remains finite for all time in our construction of front tracking approximation.  
\end{proof}
 %  Now if $f(|\mu_{1}|,|\mu_{2}|)=|\mu_{1}||\mu_{2}|(|\mu_{1}|+|\mu_{2}|)$, then under the assumption of that TV($U_{0}$)$\cdot||U_{0}||_{\f}\le\delta_{0}$ with $\delta_{0}>0$ is sufficiently small and $C$ is large enough, we get  $G(t+)\le G(t-)$. Note that $\sum\limits_{A}f(|\mu_{k}|,|\mu_{j}|)\le\left(\sum\limits_{i}|\mu_{i}|\right)^{3}$

\section{Stability estimates and proof of the main result}\label{main-estimates}
Before starting the proof of main results we recalling an important lemma from \cite{bressan2000hyperbolic}, which will be used in the proof of the main result.
\begin{lemma}[\cite{bressan2000hyperbolic}]\label{bressan-lemma}
Let $S:\mathcal{D}\times[0,\f)\rr[0,\f)$ be a Lipschitz semigroup with Lipschitz constant $L$ and let $w:[0,T]\rr\mathcal{D}$ be a Lipschitz continuous function. Then
\begin{eqnarray*}
||w(T)-S_{T}w(0)||_{L^{1}}&\le& L\int\limits_{0}^{T}\liminf_{h\rr0}\frac{||w(t+h)-S_{h}w(t)||_{L^{1}}}{h}\dy t.
\end{eqnarray*}
\end{lemma}
\begin{lemma}[ \cite{bia-col-02, bressan2000hyperbolic}]
\label{local-con-est}
Let $S_t$ be the standard Riemann semigroup generated by the following strictly hyperbolic conservation laws
\[
        w_t+F(w)_x=0 .
\]
Let
\[
        \bar w(t,x)=
        \begin{cases}
        w_-, & x<\lambda t,\\
        w_+, & x>\lambda t,
        \end{cases}
\]
be a single  front with speed $\lambda$. If the front is of shock or contact type, then
\begin{equation*}
  \liminf_{h\to0+}
        \frac{
        \|\bar w(h,\cdot)-S_h\bar w(0,\cdot)\|_{L^1}
        }{h}
        \leq
        C|\mathcal{E}|.   
\end{equation*}    
If the front is a rarefaction front, then
\begin{equation*}
 \liminf_{h\to0+}
        \frac{
        \|\bar w(h,\cdot)-S_h\bar w(0,\cdot)\|_{L^1}
        }{h}
        \leq
        C|\mathcal{E}|
        +
        C|w_+-w_-|^2,   
\end{equation*}
where
$\mathcal{E}:=
        F(w_+)-F(w_-)-\lambda(w_+-w_-)$ and  $C>0$ constant.
\end{lemma}
\begin{proof}[Proof of Theorem \ref{exi-app-p}]
To prove that there exists a weak solution to system \eqref{appr-p} for any initial data in $w_{0}\in \mathcal{D}$. We start with a piecewise constant approximation $w_{0}^m=(v_{0}^m, u_{0}^m)$ of initial data $w_{0}$. Then, from our construction, the system \eqref{appr-p} has a global piecewise constant approximate weak solution. However, from \eqref{weak-sol}, we know that the obtained approximate solution satisfies Rankine-Hugoniot condition exactly for all physical fronts and only error occurs due to the nonphysical fronts. Due to the Lemma \ref{nonphysical-smallness}, we can choose the nonphysical front parameter $h_{m}$ such that 
\[
    \sum_{\Gamma_{np}\in\mathcal{F}_{np}}
        |\mu_{np}|
        \leq
        \frac{1}{2^m},
\]
where $\mathcal{F}_{np}$ is the set of all nonphysical fronts. Let $w^{m}$ be the constructed approximate solution corresponding to initial data $w_{0}^m$.  From section \ref{uni-bv}, we know that 
\begin{align}\label{ubv}
 TV(w^m)\le C TV(w_{0}),   
\end{align}
where $C$ is a constant depends only on the fixed compact set $K$. From \eqref{ubv}, all states stay in a fixed compact set and all physical and nonphysical speeds are uniformly bounded. Hence, for all $t>s\ge0$
\begin{align}
    ||w^m(t)-w^m(s)||_{L^{1}(\R)}\le C(t-s)
\end{align}
By from Helly's theorem we can extract a subsequence such that $w^m\rr w$ in $L^1_{loc}$. Now, to prove that $w$ is weak solution we need to show that for
$\varphi\in C_{c}^{\f}(\R\times[0,T))$ it satisfies
\begin{align*}
\int\limits_{0}^{T}\int\limits_{\R}(w\varphi_{t}+F_{n}(w)\varphi_{x})\dy x\dy t+\int\limits_{\R}w(x,0)\varphi(x,0)\dy x&=0,
\end{align*}
equivalently,
\begin{align*}
\lim_{m\rr\infty}\underbrace{\int\limits_{0}^{T}\int\limits_{\R}(w^m\varphi_{t}+F_{n}(w^m)\varphi_{x})\dy x\dy t+\int\limits_{\R}w^m(x,0)\varphi(x,0)\dy x}_{I_m}&=0.
\end{align*}
To prove that $\lim_{m\rr\infty}I_{m}=0$, by Lemma \ref{entropy-thm}, it is enough to prove that Rankine-Hugoniot condition holds along all discontinuities lines. From Section \ref{app-RH}, we note that $w^m$ satisfies Rankine-Hugoniot condition exactly for all physical fronts and only error occurs due to the nonphysical fronts. Hence,
\begin{align}
    \lim_{m\rr\infty}I_{m}=\lim_{m\rr\infty}\sum_{\Gamma_{np}\in\mathcal{F}_{np}}
        |\mu_{np}|
        \leq
        \lim_{m\rr\infty}\frac{1}{2^m}=0.
\end{align}  
This proves that the $w$ is weak solution to system \eqref{appr-p}.
\end{proof}
\begin{proof}[Proof of the Theorem \ref{weak-p-syst}]
   To prove this we start with a piecewise approximation $w_{0}^n=(v_{0}^n, u_{0}^n)$ of given initial data $w_{0}$. From our construction, the system \eqref{p-sys} has a global piecewise constant approximate weak solution $w^n$. Eq. \eqref{weak-sol} provides that $w^n$ satisfies approximate Rankine-Hugoniot condition along all discontinuities. More precisely, if we define $\mathcal{E}_{n}(\gamma)$ is Rankine-Hugoniot condition along the front $\gamma$, then
   \begin{align}\label{RH-estimate}
\sum_{\gamma\in\mathcal{F}_{S}}|\mathcal{E}_{n}(\gamma)|+\sum_{\beta\in\mathcal{F}_{R}}|\mathcal{E}_{n}(\beta)|+\sum_{\Gamma_{np}\in\mathcal{F}_{np}}|\mu_{np}|&\le\frac{C}{2^n}\left(\sum_{\gamma\in\mathcal{F}_{s}}|\mu_{\gamma}|+ \sum_{\beta\in\mathcal{F}_{R}}|\mu_{\beta}|\right)+\sum_{\Gamma_{np}\in\mathcal{F}_{np}}
        |\mu_{np}|,
        \end{align}
        where $\mathcal{F}_{S}, \mathcal{F}_{R}$, and $\mathcal{F}_{np}$ are the set all shock, rarefaction and nonphysical fronts, respectively.
   Due to the Lemma \ref{nonphysical-smallness}, we can choose the nonphysical front parameter $h_{n}$ such that 
   \[\sum_{\Gamma_{np}\in\mathcal{F}_{np}}
        |\Gamma_{np}|
        \leq
        \frac{1}{2^n}.\]
        From section \ref{uni-bv}, we know that 
\begin{align}\label{ubv}
 TV(w^n)\le C TV(w_{0}),   
\end{align}
where $C$ is a constant depends only on the fixed compact set $K$. From \eqref{ubv}, all states stay in a fixed compact set and all physical and nonphysical speeds are uniformly bounded. Hence, for all $t>s\ge0$
\begin{align}
    ||w^n(t)-w^n(s)||_{L^{1}(\R)}\le C(t-s)
\end{align}
By from Helly's theorem we can extract a subsequence such that $w^n\rr w$ in $L^1_{loc}$. Now, to prove that $w$ is weak solution we need to show that for
$\varphi\in C_{c}^{\f}(\R\times[0,T))$ it satisfies
\begin{align*}
\int\limits_{0}^{T}\int\limits_{\R}(w\varphi_{t}+F(w)\varphi_{x})\dy x\dy t+\int\limits_{\R}w(x,0)\varphi(x,0)\dy x&=0,
\end{align*}
equivalently,
\begin{align*}
\lim_{n\rr\infty}\underbrace{\int\limits_{0}^{T}\int\limits_{\R}(w^m\varphi_{t}+F(w^m)\varphi_{x})\dy x\dy t+\int\limits_{\R}w^m(x,0)\varphi(x,0)\dy x}_{I_n}&=0.
\end{align*}
To prove that $\lim_{m\rr\infty}I_{n}=0$, by Lemma \ref{entropy-thm}, it is enough to prove that Rankine-Hugoniot condition holds along all discontinuities lines. Therefore, from \eqref{RH-estimate}
\begin{align}
\lim_{n\rr\infty}I_{n}&\le\lim_{n\rr\infty}\frac{C}{2^n}\left(\sum_{\gamma\in\mathcal{F}_{s}}|\mu_{\gamma}|+\sum_{\beta\in\mathcal{F}_{R}}|\mu_{\beta}|\right)+\sum_{\Gamma_{np}\in\mathcal{F}_{np}}|\mu_{np}|\\
&\le\lim_{n\rr\infty}\left(\frac{C}{2^n}+\frac{C}{2^n}\right)=0.
\end{align} 
Hence, this proves that the $w$ is weak solution to system \eqref{p-sys}. Now we prove that this weak solution is also form a standard Riemann semigroup. For that we already know that the system \eqref{p-sys} admits a SRS \cite[See Theorem 8.1]{bressan2000hyperbolic}. Let $S_{t}$ be the semigroup generated by the system \eqref{p-sys} corresponding to the initial data $w_{0}$. Now, from \cite[See Theorem 9.4]{bressan2000hyperbolic} states that $w(x,t)=S_{t}w_{0}(x)$. This completes the proof of the theorem.
\end{proof}

% In this section, we prove our main results by deriving stability estimates for the weak entropy solution of system \eqref{p-sys}, with respect to the flux function.
% \begin{itemize}
%     \item[1.] We denote the approximate solution $(v^{n}, u^{n})$ to \eqref{appr-p} which we constructed using our approximation scheme as $U^{n}:=(v^{n}, u^{n})$
%      % Let $U^{n}=(v^{n}, u^{n})$ be the approximate solution to \eqref{appr-p} obtained from our approximation scheme.
%     \item[2.] We denote the weak entropy solution to \eqref{appr-p} by $S^{n}_{t}w_{0}^{n}=w^{n}(t)$ where $w_{0}^{n}$ is a piecewise constant function in $L^{1}$. Furthermore, $S^{n}_{t}w_{0}^{n}$ satisfies Lipschitz semigroup properties (for any $w_{0}\in BV\subset L^{1}$).
%     \item[3.] We denote the weak entropy solution to \eqref{p-sys} by $S_{t}w_{0}=w(t)$ where $w_{0}\in BV\subset L^{1}$ and satisfy the Lipschitz semigroup properties (cf. \cite{bressan2000hyperbolic}).
% \end{itemize}
% Next, we introduce a notation from \cite{bia-col-02} which again needed to state the main result. Let $f$ and $g$ are two flux such that $S^{f}_{t}$ and $S^{g}_{t}$ are corresponding Lipschitz semigroup generated by these fluxes, then
% \begin{eqnarray}
% \hat{d}(f, g)=\sup_{w\in\mathcal{R}(D^{g})}\frac{1}{w_{+}-w_{-}}||S^{f}_{1}-S^{g}_{1}||_{L^{1}},
% \end{eqnarray}
% where $\mathcal{R}(D^{g})$ is a set all piecewise constant functions in $D^{f}$ having a single jump.
\begin{proof}[Proof of Theorem \ref{main-theorem}]
Our goal is to prove the estimate
\begin{eqnarray*}
||S_{t}w_{0}-w^n(t)||_{L^{1}}\le C\cdot||DF-DF_{n}||_{L^{\f}}TV(w_{0})\, t.
\end{eqnarray*}
For fix $n$, let $w_{0}^m$ is the piecewise approximation of initial data $w_{0}$, and $w_{m}^{n}$ is the approximate front tracking solution to \eqref{appr-p}. By using Lemma \ref{bressan-lemma}, we have
\begin{align}\label{semigroup-limit}
    ||S_{t}w_{0}^{m}-w^n_{m}(t)||_{L^{1}}\le L\int\limits_{0}^{t}\liminf_{h\rr0}\frac{||S_{h}w^n_m(\tau)-w^n_m(\tau+h)||_{L^{1}}}{h}\dy t.
\end{align}
Note that in \eqref{semigroup-limit} the limit does not change between two interaction times. Hence to estimate
\begin{equation*}
  \frac{||S_{h}w^n_m(\tau)-w^n_m(\tau+h)||_{L^{1}}}{h}  
\end{equation*}
it suffices to get the estimate for a single front in $w^n_{m}$, since the solution $w^n_{m}$ has only finitely many fronts. Fix a time $\tau$ which is not an interaction time. At this time,  all physical fronts of $w^n_m$ satisfy the
Rankine--Hugoniot condition for $F_n(w)=(-u, p_{n}(v))$, while nonphysical fronts give an error whose total strength tends to zero as $h\to0$. Consider a physical front with speed $\xi$ such that
\begin{eqnarray}\label{F_n}
F_{n}(w_{-})-F_{n}(w_{+})=\xi(w_{-}-w_{+}),
\end{eqnarray}
where $w_{-}$ and $w_{+}$ are the left and right states, respectively. We define  $\mathcal{E}$ for the same front with respect to $F(w)=(-u,p(v))$ is
\begin{eqnarray*}
\mathcal{E}=F(w_{-})-F(w_{+})-\xi(w_{-}-w_{+}),
\end{eqnarray*}
Using \eqref{F_n}, we obtain
\begin{align*}
        \mathcal{E}
        &=
        F(w_{-})-F(w_{+})
        -
        \bigl(F_n(w_{-})-F_n(w_{+})\bigr)
        \\
        &=
        (F-F_n)(w_{-})-(F-F_n)(w_{+}).
\end{align*}
from the fundamental theorem of calculus gives
\begin{align*}
        (F-F_n)(w_{-})-(F-F_n)(w_{+})
        &=
        \int_0^1
        (DF-DF_n)\bigl(w_{+}+\theta(w_{-}-w_{+})\bigr)
        (w_{-}-w_{+})\,d\theta .
\end{align*}
Therefore
\begin{align*}
        |\mathcal{E}|
        &\leq
        \int_0^1
        \left|
        (DF-DF_n)\bigl(w_{+}+\theta(w_{-}-w_{+})\bigr)
        \right|
        |w_{+}-w_{-}|\,d\theta\leq
        \|DF-DF_n\|_{L^\infty}\,|w_{+}-w_{-}|.
\end{align*}
Let $F(\tau)$ and $\mathcal{F}_{np}(\tau)$ denotes the set of physical fronts and set of nonphysical
fronts at $\tau$, respectively. By using the estimates from Lemma \ref{local-con-est} for each front,
\begin{align*}
\liminf_{h\to0+}\frac{||S_{h}w^n_m(\tau)-w^n_m(\tau+h)||_{L^{1}}}{h}&\leq C\|DF-DF_n\|_{L^\infty}\,|\sum_{\alpha\in\mathcal{F}(\tau)}
        |w_{+}^\alpha-w_{-}^\alpha|+C\sum_{\Gamma_{np}\in\mathcal{F}_{np}(\tau)}
        |\mu_{\Gamma_{np}}(\tau)|\\
        &\le C\|DF-DF_n\|_{L^\infty}\,\operatorname{TV}\bigl(w^{n}_{m}(\tau)\bigr)+\frac{C}{2^{m}}.
\end{align*}
From the uniform bound of total variation we get
\begin{align*}
\liminf_{h\to0+}\frac{||S_{h}w^n_m(\tau)-w^n_m(\tau+h)||_{L^{1}}}{h}
        &\le C\|DF-DF_n\|_{L^\infty}\,\operatorname{TV}\bigl(w^{m}_{0}\bigr)+\frac{C}{2^{m}}.
\end{align*}
Hence,
\begin{align*}
    ||S_{t}w_{0}^{m}-w^n_{m}(t)||_{L^{1}}\le C\|DF-DF_n\|_{L^\infty}\,\int\limits_{0}^{t}\operatorname{TV}\bigl(w^{m}_{0}\bigr)+\frac{C}{2^{m}},
\end{align*}
passing the limit $m\rr\f$, we get
\begin{eqnarray*}
||S_{t}w_{0}-w^n(t)||_{L^{1}}\le C\cdot||DF-DF_{n}||_{L^{\f}}TV(w_{0})\, t.
\end{eqnarray*}
Finally, using Lemma \ref{uniform-bound}, we conclude that
\begin{eqnarray*}
||S_{T}w_{0}-w^{n}(T)||_{L^{1}}&\le& \mathcal{O}\left(\frac{1}{2^{n}}\right)\cdot TV(w_{0})\cdot T.
\end{eqnarray*}
\end{proof}
\appendix
\section{Auxilary facts}
Here we  recall a theorem from \cite{bressan2000hyperbolic} which we used in proving the approximate entropy estimate for approximate solution constructed by the scheme.

\begin{customlemma}{A.1}\cite[Theorem 4.3]{bressan2000hyperbolic}\label{entropy-thm}
Let $u(x,t)$ be a solution of \eqref{p-sys} satisfying the piecewise Lipschitz regularity assumptions. Let $(Q,\eta)$ be a convex flux-entropy pair. Then the entropy condition 
\begin{equation*}
    \eta(u)_{t}+Q(u)_{x}\le0
\end{equation*}
holds if and only if along every jump curve $s_{j}(t)$ one have
\begin{equation*}
    Q(u^{+}_{j}(t))-Q(u^{-}_{j}(t))\le s'_{j}(t)(\eta(u^{+}_{j}(t))-\eta(u^{-}_{j}(t))),
\end{equation*}
where $u^{-}_{j}(t)$ and $u^{+}_{j}(t)$ are the left and right state at jump. 
\end{customlemma}

\begin{customlemma}{A.2}\label{uniform-bound}
Let $v\in[v_{i-1},v_{i}]$ and $u\in[v_{j-1},v_{j}]$ be in compact set $K\subset\R^{2}$, $i, j\in\mathbb{Z}$, then
\begin{equation}
    |p_{n}'(v)-p_{n}'(u)|\le L\cdot\frac{|i-j+1|}{2^{n}},
\end{equation}
where $L=\max_{K}|p''(x)|\sqrt{\frac{p'(m)}{p'(M)}}$.
\end{customlemma}
\begin{proof}
By our choice of approximation we know that
\begin{eqnarray}\label{pv0}
\left\{
\begin{array}{l}
v_{1} - v_{0} = \dfrac{1}{2^{n}}, \\[0.8ex]
(v_{1} - v_{0})\bigl(p(v_{0}) - p(v_{1})\bigr)
=
(v_{i} - v_{i-1})\bigl(p(v_{i-1}) - p(v_{i})\bigr),
\quad \text{for all } i
\end{array}
\right.
\end{eqnarray}
\eqref{pv0} gives the following for any $i$,
\begin{eqnarray}\label{vi-est}
\sqrt{\frac{p'(M)}{p'(m)}}|v_{1}-v_{0}|\le|v_{i}-v_{i-1}|\le\sqrt{\frac{p'(m)}{p'(M)}}|v_{1}-v_{0}|,
\end{eqnarray}
finally we consider,
\begin{eqnarray*}
|p_{n}'(v)-p_{n}'(u)|&=&\left|\frac{(p(v_{i-1}) - p(v_{i}))}{(v_{i} - v_{i-1})}-\frac{(p(v_{j}) - p(v_{j+1}))}{(v_{j+1} - v_{j})}\right|\\
&=&|p'(\xi)-p'(\hat{\xi})|\le\max_{K}|p''(x)||\xi-\hat{\xi}|
\end{eqnarray*}
where $\hat{\xi}\in(v_{j-1},v_{j})$, $\xi\in(v_{i-1},v_{i})$. Hence, by using \eqref{vi-est},
\begin{eqnarray*}
|p_{n}'(v)-p_{n}'(u)|&\le&L\cdot\frac{|i-j+1|}{2^{n}}.
\end{eqnarray*}
\end{proof}

\begin{customlemma}{A.3}[Uniform weighted generation bound] \label{wei-gen-bound} Fix $\chi>1$. Let $\mathcal F(t)$ denote the set of all fronts at time $t$, and $g(\gamma)$ is the generation of front $\gamma$. Define 
\[ V_{\chi}(t) := \sum_{\gamma\in\mathcal F(t)}|\mu_\gamma(t)|\,\chi^{g(\gamma)}. \]
Let \(\mathcal A(t)\) denote the set of approaching front pairs, and define 
\[ Q_{\chi}(t) := \sum_{(\alpha,\beta)\in\mathcal A(t)} |\mu_\alpha(t)\mu_\beta(t)|\, \chi^{\max\{g(\alpha),g(\beta)\}+1}. \] 
Then, for \(K>0\) sufficiently large and for sufficiently small BV initial data, the weighted functional 
\[ \Phi_\chi(t):=V_\chi(t)+KQ_\chi(t) \]
is non increasing at every interaction time. In particular, there exists a constant \(C>0\), such that \[ V_\chi(t)\leq C, \qquad\text{for every }t\geq0. \]
\end{customlemma} 
\begin{proof}
Let $\tau$ be an interaction time involving two incoming fronts \(\alpha\) and \(\beta\) with strength $|\mu_\alpha(\tau-)|$ and $|\mu_\beta(\tau-)|$, respectively. Set 
\[  g:=\max\{g(\alpha),g(\beta)\} \qquad \text{ and } \qquad
I_\chi(\alpha,\beta) := |\mu_\alpha(\tau-)\mu_\beta(\tau-)|\,\chi^{g+1}. \]
Now, we consider
\begin{align*}
\Phi_\chi(\tau+)-\Phi_\chi(\tau-):=V_\chi(\tau+)-V_\chi(\tau-)+K(Q_\chi(\tau+)-Q_\chi(\tau-)),
\end{align*}
any outgoing front at this interaction can have generation atmost $g+1$, and together with the interaction estimate, we get  
\[ \Delta V_\chi(\tau) := V_\chi(\tau+)-V_\chi(\tau-) \leq C_1 I_\chi(\alpha,\beta), \]
 for some constant $C_1>0$. We now estimate the change of \(Q_w\). Since \((\alpha,\beta)\) is an approaching physical pair, its contribution $I_\chi(\alpha,\beta)$ is removed from $Q_\chi$ at the interaction. The outgoing fronts may create new approaching pairs with other physical fronts, and they are bounded by $C_2 V_\chi(\tau-) I_\chi(\alpha,\beta).$
  Thus \[ \Delta Q_\chi(\tau) := Q_\chi(\tau+)-Q_\chi(\tau-) \leq - I_\chi(\alpha,\beta) + C_2 V_\chi(\tau-) I_\chi(\alpha,\beta). \]
  By using the smallness of initial data total variation and choosing $K$ sufficiently large, we get
  \begin{align*}
\Phi_\chi(\tau+)\le\Phi_\chi(\tau-)\le \Phi_\chi(0).
  \end{align*}
We next consider a modified same-family shock--rarefaction interaction. Let the incoming physical fronts again be denoted by 
$\alpha,\beta$. In the simplified solver the opposite family front is replaced by one nonphysical front 
$\Gamma_{\mathrm{np}}$ and $g(\Gamma_{\mathrm{np}})=g+1$. By interaction estimate, we get  
\[ |\mu_{\Gamma_{\mathrm{np}}}|\chi^{g+1}\le\Delta V_\chi(\tau)\le C_1 I_\chi(\alpha,\beta), \]
for some constant $C_1>0$. The incoming shock--rarefaction pair is an approaching physical pair, and its weighted contribution \(I_\chi(\alpha,\beta)\) is removed from \(Q_\chi\). The same estimate as above gives \[ \Delta Q_\chi(\tau) \leq -I_\chi(\alpha,\beta) + C_2V_\chi(\tau-)I_\chi(\alpha,\beta). \] Using the smallness condition and choosing $K$ sufficiently large
we get
  \begin{align*}
\Phi_\chi(\tau+)\le\Phi_\chi(\tau-)\le \Phi_\chi(0).
  \end{align*}
   Finally, we consider an interaction between a nonphysical front and a physical front. They cross each other and the nonphysical front keeps both its strength and its generation, and the physical front is unchanged. Hence $\Delta V_\chi(\tau)=0$. Moreover, \(Q_\chi\) contains only physical--physical approaching pairs. Since no physical front is changed and no new physical front is created, we also have $\Delta Q_\chi(\tau)=0$.  Thus \[ \Delta\Phi_\chi(\tau)=0. \] We have therefore shown that \(\Phi_\chi\) is non increasing at every interaction time.
  Since all initial fronts have generation $1$, one has $Q_\chi(0)\le V_\chi(0)^2$. Hence,
\begin{align*}
    V_\chi(t)\le\Phi_\chi(t)\le V_\chi(0)+KV_\chi(0)^2\le C,
\end{align*}
  where constant $C$ is depending only on initial data and $K$. This completes the proof. \end{proof}

\section{Explicit example for infinite interaction}\label{inf-front-accu}
For simplicity, we choose the $\gamma=1$ pressure laws $p(v)=\frac{1}{v^\gamma}$, one can make the example for any $\gamma>1$. We define the polygonal approximation of the pressure $p(v)=\frac{1}{v},$ for $v>0$,
\begin{equation} p_{n}(v)=\frac{v-v_{i}^{n}}{v_{i+1}^{n}-v_{i}^{n}}p\left(v_{i+1}^{n}\right)+\frac{v_{i+1}^{n}-v}{v_{i+1}^{n}-v_{i}^{n}}p\left(v_{i}^{n}\right) \quad \mbox{ for } v\in\left[v_{i}^{n}, v_{i+1}^{n}\right]. \end{equation} where $v_{i}^{n}\in\mathcal{V}=\{v_{i}^{n}: (v_{i+1}^{n}-v_{i}^{n})(p(v_{i}^{n})-p(v_{i+1}^{n}))=(v_{1}^{n}-v_{0}^{n})(p(v_{0}^{n})-p(v_{1}^{n}))\}$, for all $i\in\mathbb{Z}$, and we choose for some fix $n\in\mathbb{N}$ 
\begin{equation*}
v_0^n=1,\qquad v_1^n=1+\frac1{2^n}:=a_n.
\end{equation*}
we can see that for $p(v)=\frac{1}{v}$, we get $v_i^n:=a_n^i$ for $i\in\mathbb{Z}$.
Since $p(v_i^n)=a_n^{-i}$, we may write
\begin{align}
     p_n'(v)
        &=\frac{p(v_{i+1}^n)-p(v_i^n)}
        {v_{i+1}^n-v_i^n}
        =\frac{a_n^{-i-1}-a_n^{-i}}
        {a_n^{i+1}-a_n^i}=
        -a_n^{-2i-1}.
\end{align}
For choosing initial data we consider two points $x_0< x_1$ such that
 \begin{align}
  (v^n,u^n)(0,x)=
        \begin{cases}
        (v_l,u_l), & x<x_0,\\
        (v_m,u_m), & x_0<x<x_1,\\
        (v_r,u_r), & x>x_1,
        \end{cases}   
 \end{align}
 where for a fix integer $k\ge4$ we choose \[
        v_l:=v_k^n=a_n^k,\qquad 
        u_l=1,\qquad v_m:=v_0^n=1,\qquad
       % V_1:=v_1^n=a_n,\qquad  
       v_r:=v_2^n=a_n^2,
\]
now we choose $u_{m}, u_{r}$, such that the Riemann problem at $x_{0}$ have $1$-family shock, denoted as $\beta$, and the Riemann problem at $x_{1}$ consists of two $1$-family shock rarefaction fronts, denoted as $\gamma_{1}, \gamma_{2}$, thus 
\[
        \beta:\ (v_{l}, u_{l})\to (v_m,u_m),\qquad
        \gamma_1:\ (v_m,u_m)\to (v_m^*,u_m^*) ,\qquad
        \gamma_2:\ (v_m^*,u_m^*)\to (v_r,u_r),
\]
where $v_m^*:=v_{1}^{n}=a_{n}$, which directly comes from the piecewise affine approximation of $p(v)$. Let $c_{\beta}, c_{1}$, and $c_{2}$ are the speed of $\beta, \gamma_{1}$, and $\gamma_{2}$, respectively. 

\begin{align}
        c_\beta
        &:=
        -\sqrt{\frac{p_n(v_m)-p_n(v_l)}{v_l-v_m}}
        =
        -\sqrt{\frac{1-a_n^{-k}}{a_n^k-1}}
        =
        -a_n^{-k/2},                                      \\
        c_1&:=-a_n^{-1/2},\qquad
        c_2:=-a_n^{-3/2}.
\end{align}
Since $k\ge4$, we have $c_1<c_2<c_\beta<0$. We have fixed the $u_l=1$, hence we can choose $u_{m}, u_m^*$, and $u_r$ by the Rankine-Hugoniot condition
\begin{align}
        u_m-u_l&=-c_\beta(v_m-v_l),\\
        u_m^*-u_m&=-c_1(v_m^*-v_m),\\
       u_r-u_m^*&=-c_2(v_r-v_m^*).
\end{align}
Since $x_{0}<x_{1}$, and $\beta$ is located at $x_0$,  and $\gamma_1, \gamma_2$ at $x_1$, hence from there spatial order and speed order $c_1<c_2<c_\beta<0$, we know that $\beta$ and $\gamma_1$ will interact first at $\tau_{1}>0$. To describe the interaction pattern, set $v_*:=v_m^*=v_1^n=a_n$, $c_1:=a_n^{-1/2}$, and $c_2:=a_n^{-3/2}$. Let $\xi_1:=v_m$, $\eta_1:=u_m$, $\zeta_1:=v_r$, and $\omega_1:=u_r$. Suppose that before the $j$-th interaction the fronts are
\[
        \beta_j:\ (v_l,u_l)\to(\xi_j,\eta_j),\qquad
        \gamma_1^j:\ (\xi_j,\eta_j)\to(v_*,u_j^*),\qquad
        \gamma_2^j:\ (v_*,u_j^*)\to(\zeta_j,\omega_j),
\]
where $\xi_j\in[v_m,v_*]$, $\zeta_j\in[v_*,v_r]$, and
\begin{align}
\eta_j-u_l&=-c_{\beta_j}(\xi_j-v_l)\\
        u_j^*-\eta_j&=-c_1(v_*-\xi_j),\\
        \omega_j-u_j^*&=-c_2(\zeta_j-v_*).
\end{align} 
Thus $\gamma_1^j$ and $\gamma_2^j$ are $1$-rarefaction fronts with speeds $-c_1$ and $-c_2$, respectively.

At the interaction between $\beta_j$ and $\gamma_1^j$, the outgoing fronts are
\[
        \beta_j+\gamma_1^j\longrightarrow \beta_{j+1}+\delta_j,
\]
where
\[
        \beta_{j+1}:\ (v_l,u_l)\to(\xi_{j+1},\eta_{j+1}),\qquad
        \delta_j:\ (\xi_{j+1},\eta_{j+1})\to(v_*,u_j^*).
\]
Here $\xi_{j+1}$, and $\eta_{j+1}$, is determined by the Rankine--Hugoniot conditions
\begin{align}
        \eta_{j+1}-u_l
        &=-c_{\beta_{j+1}}(\xi_{j+1}-v_l),\\
        u_j^*-\eta_{j+1}
        &=-c_1(v_*-\xi_{j+1}),
\end{align}
where
$c_{\beta_{j}}:=-\sqrt{\frac{p_n(\xi_j)-p_n(v_l)}{v_l-\xi_j}}$. Next, we will show that $\xi_{j+1}\in(\xi_j,v_*)$. Let
\[
        L:=v_l=a_n^k,\qquad v_*:=v_1^n=a_n,\qquad %c_1:=a_n^{-1/2}.
\]
We define $D(\xi):=\sqrt{\bigl(p_n(\xi)-p_n(v_l)\bigr)(v_l-\xi)}$, for $\xi\in[1,\infty)$. The $1$-shock front  $\beta_{j+1}:\ (v_l,u_l)\to(\xi_{j+1},\eta_{j+1})$ given by
\[
        \eta_{j+1}=u_l-D(\xi_{j+1}),
\]
and the outgoing \(2\)-shock $\delta_j:\ (\xi_{j+1},\eta_{j+1})\to(v_*,u_j^*)$ is given by
\[
        \eta_{j+1}=u_j^*+C_1(v_*-\xi_{j+1}).
\]
Now we define, 
\[ \Phi_j(\xi):=u_l-D(\xi)-u_j^*-c_1(v_*-\xi),
        \qquad \xi\in[\xi_j,\infty)\]
and \(\xi_{j+1}\) is the zero of $\Phi_j(\xi)$.
Since \(\gamma_1^j:(\xi_j,\eta_j)\to(v_*,u_j^*)\) is a \(1\)-rarefaction front, we have
\[
        u_j^*-\eta_j=C_1(v_*-\xi_j),
        \qquad
        \eta_j=u_l-D(\xi_j).
\]
Hence
\[
        \Phi_j(\xi_j)=-2c_1(v_*-\xi_j)<0.
\]
We now check the sign at \(v_*\). Since \(p_n\) is affine on \([1,a_n]\), one has
\[
        p_n(\xi)-p_n(v_l)
        =
        \frac{a_n+1}{a_n}-a_n^{-k}-\frac{\xi}{a_n}.
\]
A direct differentiation gives
\[
        -D'(\xi)
        =
        \frac{L+a_nB-2\xi}{2a_nD(\xi)},
        \qquad
        B:=\frac{a_n+1}{a_n}-a_n^{-k}.
\]
Moreover,
\begin{align}
        (L+a_nB-2\xi)^2-4a_nD(\xi)^2
        &=(L-a_nB)^2>0.
\end{align}
Therefore \(-D'(\xi)>c_1\) for \(\xi\in[\xi_j,v_*]\). Consequently,
\begin{align}
        \Phi_j(v_*)
        &=
        D(\xi_j)-D(v_*)-C_1(v_*-\xi_j)        \notag\\
        &=
        \int_{\xi_j}^{v_*}\bigl(-D'(\xi)-C_1\bigr)\,d\xi>0.
\end{align}
Thus \(\Phi_j(\xi_j)<0<\Phi_j(v_*)\). Hence there exists
\begin{align}\label{xi_j+1}
   \xi_{j+1}\in(\xi_j,v_*). 
\end{align}
Since also
\[
        \Phi_j'(\xi)=-D'(\xi)+C_1>0,
\]
this zero is unique. Therefore the right state of the outgoing \(1\)-shock remains strictly inside the grid interval \([v_m,v_*]\). We now describe the interaction between \(\delta_j\) and \(\gamma_2^j\).
Recall that
\begin{align}
  \delta_j:\ (\xi_{j+1},\eta_{j+1})\to(v_*,u_j^*),
        \qquad
        \gamma_2^j:\ (v_*,u_j^*)\to(\zeta_j,\omega_j), 
\end{align}

where $\zeta_j\in[v_*,v_r]$. The incoming states satisfy
\begin{align}\label{incoming}
        u_j^*-\eta_{j+1}&=-c_1(v_*-\xi_{j+1}),\\\label{incoming-2}
        \omega_j-u_j^*&=c_2(\zeta_j-v_*).
\end{align}
We claim that the outgoing \(1\)-wave necessarily crosses the grid point \(v_*\), and hence splits there. In that case the outgoing fronts are
\[
        \delta_j+\gamma_2^j
        \longrightarrow
        \gamma_1^{j+1}+\gamma_2^{j+1}+\delta_j',
\]
where
\[
        \gamma_1^{j+1}:\ (\xi_{j+1},\eta_{j+1})\to(v_*,u_{j+1}^*),
\]
\[
        \gamma_2^{j+1}:\ (v_*,u_{j+1}^*)\to(\zeta_{j+1},\omega_{j+1}),
        \qquad
        \delta_j':\ (\zeta_{j+1},\omega_{j+1})\to(\zeta_j,\omega_j).
\]
We calculate the $\zeta_{j+1}$ from Rankine--Hugoniot relations for the outgoing fronts
\begin{align}\label{outgoing}
        u_{j+1}^*-\eta_{j+1}&=c_1(v_*-\xi_{j+1}),\\\label{outgoing2}
        \omega_{j+1}-u_{j+1}^*&=c_2(\zeta_{j+1}-v_*),\\\label{outgoing 3}
        \omega_j-\omega_{j+1}&=-c_2(\zeta_j-\zeta_{j+1}).
\end{align}
 By using \eqref{incoming}-\eqref{incoming-2}, we get
 \begin{align*}
     \omega_j=\eta_{j+1}+c_2(\zeta_j-v_*)-c_1(v_*-\xi_{j+1}),
 \end{align*}
and by using \eqref{outgoing}-\eqref{outgoing2}, we get
\begin{align*}
     \omega_{j+1}=\eta_{j+1}+c_1(v_*-\xi_{j+1})+c_2(\zeta_{j+1}-v_*),
 \end{align*}
 now by substituting these in \eqref{outgoing 3},
\begin{align}\label{recur}
        \zeta_{j+1}
        &=\zeta_j-\frac{c_1}{c_2}(v_*-\xi_{j+1})
          =\zeta_j-a_n(v_*-\xi_{j+1}),
\end{align}
since we assume that splitting occured at $j$-th interaction, thus $\zeta_{j+1}>v_*$, then we get
\begin{align}
  a_n(v_*-\xi_{j+1})<\zeta_j-v_*.  
\end{align}
We now prove that this inequality holds for every \(j\). For simplicity choose \(k=4\). Set
\[
        x_j:=v_*-\xi_j,\qquad y_j:=\zeta_j-v_* .
\]
Moreover, whenever the splitting occurs at the \(j\)-th step, \eqref{recur} gives
\[
        y_{j+1}=y_j-a_nx_{j+1}.
\]
We first estimate \(x_{j+1}\) in terms of \(x_j\). Recall that \(\xi_{j+1}\) is determined by
\[
        D(\xi_j)-D(\xi_{j+1})=c_1(x_j+x_{j+1}),
\]
where
\[
        D(\xi):=\sqrt{\bigl(p_n(\xi)-p_n(v_l)\bigr)(v_l-\xi)}.
\]
Since \(k=4\), \(v_l=a_n^4\), and for \(\xi\in[1,a_n]\),
\[
        p_n(\xi)-p_n(v_l)
        =
        \frac{a_n+1}{a_n}-a_n^{-4}-\frac{\xi}{a_n}.
\]
A direct computation gives
\[
        -D'(\xi)
        =
        \frac{a_n^4+a_nB-2\xi}{2a_nD(\xi)},
        \qquad
        B:=\frac{a_n+1}{a_n}-a_n^{-4},
\]
and
\[
        \frac{d}{d\xi}\bigl(-D'(\xi)\bigr)
        =
        \frac{(a_n^4-a_nB)^2}{4a_n^2D(\xi)^3}>0.
\]
Hence \(-D'\) is increasing on \([1,a_n]\). Therefore
\[
        -D'(\xi)\le -D'(a_n)
        =
        \frac{a_n^4+1}{2a_n^{5/2}}
        =:M_n .
\]
It follows that
\[
        D(\xi_j)-D(\xi_{j+1})
        \le M_n(\xi_{j+1}-\xi_j)
        =
        M_n(x_j-x_{j+1}).
\]
Combining this with \(D(\xi_j)-D(\xi_{j+1})=c_1(x_j+x_{j+1})\), we obtain
\[
        x_{j+1}
        \le
        \frac{M_n-c_1}{M_n+c_1}\,x_j.
\]
Since \(c_1=a_n^{-1/2}\), this gives
\[
        x_{j+1}\le \lambda_n x_j,
        \qquad
        \lambda_n:=
        \frac{M_n-c_1}{M_n+c_1}
        =
        \left(\frac{a_n^2-1}{a_n^2+1}\right)^2.
\]
In particular, \(0<\lambda_n<1/2\). We now prove the splitting by induction. For \(j=1\), \(x_1=v_*-v_m=a_n-1\), while
\[
        y_1=v_r-v_*=a_n^2-a_n=a_nx_1,
\]
\[
        a_nx_2\le a_n\lambda_n x_1<a_nx_1=y_1,
\]
so the outgoing \(1\)-wave crosses \(v_*\), and the first splitting occurs.

Assume now that the splitting has occurred up to the \(j\)-th step. Then
\[
        y_j
        =
        y_1-a_n\sum_{\ell=2}^{j}x_\ell .
\]
Using \(x_\ell\le \lambda_n^{\ell-1}x_1\), we obtain
\[
        y_j-a_nx_{j+1}
        \ge
        a_nx_1
        \left(
        1-\sum_{\ell=1}^{j}\lambda_n^\ell
        \right).
\]
Since \(\lambda_n<1/2\),
\[
        \sum_{\ell=1}^{\infty}\lambda_n^\ell
        =
        \frac{\lambda_n}{1-\lambda_n}<1.
\]
Hence
\[
        y_j-a_nx_{j+1}>0,
\]
that is,
\[
        a_n(v_*-\xi_{j+1})<\zeta_j-v_*.
\]
Therefore the outgoing \(1\)-wave crosses \(v_*\) at the \((j+1)\)-th step as well. Consequently the splitting continues for all \(j\), and
\[
        v_*<\zeta_{j+1}<\zeta_j\le v_r
        \qquad\text{for every }j\ge1.
\]
Thus the right endpoint of the second \(1\)-rarefaction front also changes. In particular, whenever
\[
        a_n(v_*-\xi_{j+1})<\zeta_j-v_*,
\]
we have
\[
        v_*<\zeta_{j+1}<\zeta_j.
\]
Hence \(\gamma_2^{j+1}\) remains inside the same grid interval \([v_*,v_r]\), but its right endpoint moves into the interior. We now prove that the above infinite interaction pattern accumulates in finite time.
Let \(t_j\) denote the time at which the pair \(\gamma_1^j,\gamma_2^j\) is created, and let \(\tau_j\) denote the time at which \(\beta_j\) interacts with \(\gamma_1^j\). We set
\[
        s_j:=\tau_j-t_j.
\]
For \(j=1\), we have \(t_1=0\), and since \(\beta_1\) starts from \(x_0\) while \(\gamma_1^1,\gamma_2^1\) start from \(x_1\), 
\begin{equation}\label{s1}
        s_1=\frac{x_1-x_0}{c_{\beta_1}+c_1}>0.
\end{equation}
Here \(c_{\beta_1}\) is the speed of \(\beta_1\).

At time \(\tau_j\), the distance between \(\gamma_1^j\) and \(\gamma_2^j\) is
\[
        (c_1-c_2)s_j,
\]
because their signed speeds are \(-c_1\) and \(-c_2\). Let \(b_j>0\) be the speed of the outgoing \(2\)-shock \(\delta_j\). Then \(\delta_j\) catches \(\gamma_2^j\) after the time
\begin{equation}\label{dj}
        d_j:=t_{j+1}-\tau_j
        =
        \frac{c_1-c_2}{b_j+c_2}s_j .
\end{equation}

Let \(\widetilde c_{\beta_j}\) be the signed speed of the outgoing \(1\)-shock during the time interval \([\tau_j,t_{j+1}]\), and let \(c_{\beta_{j+1}}\) be the signed speed of \(\beta_{j+1}\) during \([t_{j+1},\tau_{j+1}]\). From the fact that the right state of every \(\beta_j\) lies in \([v_m,v_*]\), we have
\begin{equation}\label{beta-speed-lower}
        c_{\beta_j}\ge -a_n^{-k/2}>-c_2,
        \qquad
        \widetilde c_{\beta_j}\ge -a_n^{-k/2}>-c_2 .
\end{equation}
At time \(t_{j+1}\), the new front \(\gamma_1^{j+1}\) is created at the position of \(\delta_j\). Hence its distance from the shock \(\beta_{j+1}\) is
\[
        (b_j-\widetilde c_{\beta_j})d_j .
\]
Therefore
\begin{align}
        s_{j+1}
        &=
        \frac{b_j-\widetilde c_{\beta_j}}{c_{\beta_{j+1}}+c_1}d_j \notag\\
        &=
        \frac{c_1-c_2}{b_j+c_2}
        \frac{b_j-\widetilde c_{\beta_j}}{c_{\beta_{j+1}}+c_1}s_j .
        \label{sj-recur}
\end{align}
Using \eqref{beta-speed-lower}, we get
\[
        b_j-\widetilde c_{\beta_j}<b_j+c_2,
        \qquad
        c_{\beta_{j+1}}+c_1\ge c_1-a_n^{-k/2}.
\]
Thus \eqref{sj-recur} gives
\begin{equation}\label{sj-contraction}
        s_{j+1}\le \kappa s_j,
        \qquad
        \kappa:=\frac{c_1-c_2}{c_1-a_n^{-k/2}}.
\end{equation}
Since \(k\ge4\), we have \(-a_n^{-k/2}>-c_2\), or equivalently \(a_n^{-k/2}<c_2\). Hence
\[
        0<c_1-c_2<c_1-a_n^{-k/2},
\]
and therefore
\begin{equation}\label{kappa-less-one}
        0<\kappa<1.
\end{equation}
By induction from \eqref{sj-contraction},
\begin{equation}\label{sj-bound}
        s_j\le \kappa^{j-1}s_1.
\end{equation}
 From \eqref{dj} and \(b_j>0\), we have
\begin{equation}\label{dj-bound}
        d_j
        =
        \frac{c_1-c_2}{b_j+c_2}s_j
        \le
        \frac{c_1-c_2}{c_2}s_j .
\end{equation}
Set
\[
        \mu:=\frac{c_1-c_2}{c_2}.
\]
Then \(d_j\le\mu s_j\), and using \eqref{sj-bound}, we obtain
\begin{align}
        t_{j+1}-t_j
        &=s_j+d_j
        \le (1+\mu)s_j
        \le (1+\mu)\kappa^{j-1}s_1 .
        \label{time-step-bound}
\end{align}
Therefore
\begin{align}
        \sum_{j=1}^{\infty}(t_{j+1}-t_j)
        &\le
        (1+\mu)s_1\sum_{j=1}^{\infty}\kappa^{j-1}
        =
        \frac{(1+\mu)s_1}{1-\kappa}<\infty.
        \label{finite-time-sum}
\end{align}
Hence the sequence \((t_j)\) converges to a finite time
\[
        T_*:=\lim_{j\to\infty}t_j<\infty.
\]
Since each cycle contains the two interactions \(\beta_j+\gamma_1^j\) and \(\delta_j+\gamma_2^j\), infinitely many interactions occur before \(T_*\).

\begin{center}
\begin{tikzpicture}[scale=1,every node/.style={font=\tiny}]

% axes
\draw[->] (-.8,0)--(6.4,0);
\draw[->] (0,-.3)--(0,7.1);
\draw (6.55,0) node{$x$};
\draw (-.25,7.15) node{$t$};

% points
\coordinate (X0) at (3.55,0);
\coordinate (X1) at (5.67,0);

\coordinate (A1) at (3.00,1.00);
\coordinate (B1) at (3.97,2.49);

\coordinate (A2) at (2.10,3.19);
\coordinate (B2) at (2.78,4.23);

\coordinate (A3) at (1.57,4.68);
\coordinate (B3) at (2.01,5.35);

\coordinate (A4) at (1.26,5.63);
\coordinate (B4) at (1.53,6.05);

\coordinate (A5) at (1.16,6.19);
\coordinate (B5) at (1.29,6.40);

\coordinate (A6) at (1.13,6.46);
\coordinate (B6) at (1.19,6.54);

\coordinate (Ainf) at (1.12,6.65);

% extensions of S_2 fronts, parallel to each other
\coordinate (B1e) at (4.05,2.61);
\coordinate (B2e) at (2.86,4.35);
\coordinate (B3e) at (2.09,5.47);
\coordinate (B4e) at (1.61,6.17);
\coordinate (B5e) at (1.37,6.52);
\coordinate (B6e) at (1.27,6.66);

% time guide lines
\draw[dashed] (-.55,1.00)--(A1);
\draw[dashed] (-.55,2.49)--(B1);
\draw[dashed] (-.55,3.19)--(A2);
\draw[dashed] (-.55,4.23)--(B2);
\draw[dashed] (-.55,4.68)--(A3);
\draw[dashed] (-.55,5.35)--(B3);
\draw[dashed] (-.55,5.63)--(A4);
\draw[dashed] (-.55,6.65)--(Ainf);

\draw (-.82,1.00) node{$\tau_1$};
\draw (-.80,2.49) node{$t_2$};
\draw (-.82,3.19) node{$\tau_2$};
\draw (-.80,4.23) node{$t_3$};
\draw (-.82,4.68) node{$\tau_3$};
\draw (-.80,5.35) node{$t_4$};
\draw (-.82,5.63) node{$\tau_4$};
\draw (-.82,6.65) node{$T_*$};

% S_1 shock beta_j, speed increasing after each interaction
\draw[red] (X0)--(A1);
\draw[red] (A1)--(A2);
\draw[red] (A2)--(A3);
\draw[red] (A3)--(A4);
\draw[red] (A4)--(A5);
\draw[red] (A5)--(A6);
\draw[red] (A6)--(Ainf);

% rightmost rarefaction line gamma_2^j, same speed
\draw[blue] (X1)--(Ainf);

% gamma_1^j fronts, parallel
\draw[blue] (X1)--(A1);
\draw[blue] (B1)--(A2);
\draw[blue] (B2)--(A3);
\draw[blue] (B3)--(A4);
\draw[blue] (B4)--(A5);
\draw[blue] (B5)--(A6);

% S_2 shocks, parallel and extended beyond gamma_2 line
\draw[red] (A1)--(B1e);
\draw[red] (A2)--(B2e);
\draw[red] (A3)--(B3e);
\draw[red] (A4)--(B4e);
\draw[red] (A5)--(B5e);
\draw[red] (A6)--(B6e);

% nodes
\foreach \P in {X0,X1,A1,B1,A2,B2,A3,B3,A4,B4,A5,B5,A6,B6,Ainf}
{
    \fill (\P) circle (0.35pt);
}

% x labels
\draw (X0)+(0,-.35) node{$x_0$};
\draw (X1)+(0,-.35) node{$x_1$};

% front labels
\draw (3.1,.55) node{$\beta_1$};
\draw (2.2,2.2) node{$\beta_2$};
\draw (1.6,4) node{$\beta_3$};

\draw (4.45,.80) node{$\gamma_1^1$};
\draw (3.05,3.10) node{$\gamma_1^2$};
\draw (2.10,4.7) node{$\gamma_1^3$};

\draw (5.20,1.10) node{$\gamma_2^1$};
\draw (3.7,3.20) node{$\gamma_2^2$};
\draw (2.5,4.90) node{$\gamma_2^3$};

\draw (3.55,2.1) node{$\delta_1$};
\draw (2.35,3.85) node{$\delta_2$};
\draw (1.72,5.2) node{$\delta_3$};

\end{tikzpicture}

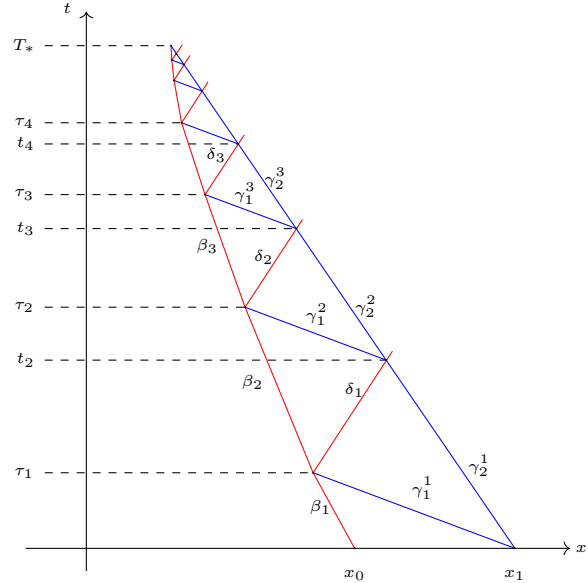
\captionof{figure}{Illustration of infinite interaction of the \(S_1\)-shock fronts \(\beta_j\), which interact successively with the \(1\)-rarefaction fronts and accumulate at time \(T_*\).}
\label{fig:xt-cascade}
\end{center}

\noindent\textbf{Acknowledgement.}
%PG was supported by  the National Science Center (Poland), project 2023/49/ B/ST1/02797. 
Author was supported by the ``Excellence Initiative Research University (IDUB)'' program at the University of Warsaw (Poland) and National Science Center (Poland). 
%A\'SG was supported by   the National Science Center (Poland), project 2023/50/A/ ST1/00447.

\end{document}